\documentclass[a4paper, 10pt]{amsart}

\usepackage{amsmath, amsfonts, amssymb}
\usepackage{amsthm}
\usepackage{esint}
\usepackage{tikz} 
\usetikzlibrary{matrix}
\usepackage[
bookmarks=false]{hyperref}
\usepackage{cite}
\usepackage{array}
\usepackage{breqn}
\usepackage{mathrsfs}
\usepackage{amsaddr}
\usepackage{enumerate}
\usepackage{pgfplots}
\usepackage [english]{babel}
\usepackage [autostyle, english = american]{csquotes}
\MakeOuterQuote{"}

\numberwithin{equation}{section}

\title{The core of the Levi distribution}

\author{Gian Maria Dall'Ara}
\address{Istituto Nazionale di Alta Matematica ``F. Severi"\\ Research Unit Scuola Normale Superiore\\
Piazza dei Cavalieri, 7, 56126, Pisa (Italy)}
\email{dallara@altamatematica.it}

\author{Samuele Mongodi}
\address{Politecnico di Milano\\ Piazza Leonardo da Vinci, 32, I-20133, Milano (Italy)
}
\email{samuele.mongodi@polimi.it}
\thanks{}

\date{\today}

\newcommand{\C}{\mathbb{C}}
\newcommand{\R}{\mathbb{R}}

\newcommand{\Z}{\mathbb{Z}}
\newcommand{\N}{\mathbb{N}}

\newcommand{\dbar}{\overline{\partial}}
\newcommand{\DF}{\mathrm{DF}}
\newcommand{\norm}{\mathfrak{n}}

\newcommand{\levinull}{\mathcal{N}}
\newcommand{\distr}{\mathcal{D}}
\newcommand{\core}{\mathfrak{C}}
\newcommand{\sphere}{\mathbb{S}}

\newcommand{\PH}{\mathrm{PH}}

\newcommand{\OO}{\mathcal{O}}

\newtheorem{thm}{Theorem}[section]
\newtheorem{prp}[thm]{Proposition}
\newtheorem{lem}[thm]{Lemma}
\newtheorem{cor}[thm]{Corollary}
\newtheorem{prp_dfn}[thm]{Proposition-Definition}
\newtheorem{dfn}[thm]{Definition}
\newtheorem{ex}[thm]{Example}
\newtheorem*{rmk}{Remark}
\newtheorem*{rmk2}{Historical remark}

\makeatletter
\renewcommand\subsection{\@startsection{subsection}{2}%
  \z@{-.5\linespacing\@plus-.7\linespacing}{.5\linespacing}%
  {\normalfont\scshape}}
\makeatother

\begin{document}
	
\maketitle

\begin{abstract}
We introduce a new geometrical invariant of CR manifolds of hypersurface type, which we dub the "Levi core" of the manifold. When the manifold is the boundary of a smooth bounded pseudoconvex domain, we show how the Levi core is related to two other important global invariants in several complex variables: the Diederich--Fornæss index and the D'Angelo class (namely the set of D'Angelo forms of the boundary). We also show that the Levi core is trivial whenever the domain is of finite-type in the sense of D'Angelo, or the set of weakly pseudoconvex points is contained in a totally real submanifold, while it is nontrivial if the boundary contains a local maximum set. 

As corollaries to the theory developed here, we prove that for any smooth bounded pseudoconvex domain with trivial Levi core the Diederich--Fornæss index is one and the $\dbar$-Neumann problem is exactly regular (via a result of Kohn and its generalization by Harrington).

Our work builds on and expands recent results of Liu and Adachi--Yum. 
\end{abstract}

\tableofcontents

\section{Introduction}

Let $\Omega\subset\C^n$ be a smooth pseudoconvex domain, and let $M$ be its boundary. The purpose of this paper is to define a distribution $\mathfrak{C}$ of complex subspaces: \[p\in M \quad\longmapsto\quad\mathfrak{C}_p\subset\C \otimes T_pM,\] which we call the "Levi core" of $M$, and to relate it to two important global invariants in analysis in several complex variables: the Diederich--Fornaess index of $\Omega$ and the D'Angelo class of $M$, i.e., the set of D'Angelo forms of $M$ (see below for precise definitions).

\medskip

The Levi core $\mathfrak{C}$ may be obtained as a special case of a general differential-geometric construction, which we describe first. If $\mathcal{D}=\{\mathcal{D}_p\}_{p\in M}$ is a distribution of subspaces on a real smooth manifold $M$, that is, $\mathcal{D}_p$ is a linear subspace of the tangent space $T_pM$ for every $p\in M$, then we define the \textbf{derived distribution} $\mathcal{D}'=\{\mathcal{D}_p'\}_{p\in M}$ as \[
\mathcal{D}_p':=\mathcal{D}_p\cap T_pS_{\mathcal{D}},
\] where: \begin{enumerate}
    \item $S_{\mathcal{D}}$ is the support of $\mathcal{D}$, namely the set of points $p\in M$ such that the fiber $\mathcal{D}_p$ has positive dimension;
    \item $T_pS_{\mathcal{D}}\subset T_pM$ is the "$C^\infty$ Zariski" tangent space to the subset $S_{\mathcal{D}}\subset M$ at the point $p$ (since we do not make any regularity assumption on $\mathcal{D}$, its support $S_{\mathcal{D}}$ need not be a submanifold and we need a notion of tangent space valid for general subsets, see Section \ref{distribution_sec} for details).
\end{enumerate}
In other words, the vectors of the distribution $\mathcal{D}$ that survive in the derived distribution $\mathcal{D}'$ are those that are also tangent to the support of $\mathcal{D}$. In this way, we obtain a smaller distribution $\mathcal{D}'\subseteq \mathcal{D}$ (i.e., $\mathcal{D}'_p\subseteq \mathcal{D}_p$ for every $p$). The operation can be iterated, yielding a decreasing sequence of distributions \[
\mathcal{D}\supseteq \mathcal{D}^{(1)}=\mathcal{D}'\supseteq \mathcal{D}^{(2)}=(\mathcal{D}')'\supseteq\ldots\]
This sequence does not necessarily stabilize after finitely many steps, but it must do so "eventually". More precisely, if one defines $\mathcal{D}^{(\alpha)}$ for any ordinal $\alpha$ by transfinite recursion, it turns out that there exists a countable ordinal $\alpha_1$ such that $\mathcal{D}^{(\alpha_1)}$ equals its derived distribution (Theorem \ref{Cantor_Bendixson}).
In fact, this result (and the whole construction of "iterated derived distributions") is a generalization of the very classical Cantor--Bendixson theorem in set theory (see Section \ref{core_sec}). This "stable" distribution $\mathcal{D}^{(\alpha_1)}$ will be called the \textbf{core} of the distribution $\mathcal{D}$, and denoted by $\mathfrak{C}(\mathcal{D})$.  Notice that the construction has an obvious generalization to the case of complex distributions of subspaces $\mathcal{D}=\{\mathcal{D}_p\}_{p\in M}$, that is, when $\mathcal{D}_p$ is a complex linear subspace of the complexified tangent space $\C\otimes T_pM$ for every $p$ (and $M$ is still a real manifold). It is this complex analogue that we need to define the Levi core.

\medskip

\par Let then $M$ be the boundary of a smooth pseudoconvex domain $\Omega\subset\C^n$ (or, more generally, let $M$ be a pseudoconvex CR manifold of hypersurface type). Denote by $T^{1,0}M\subseteq \C\otimes TM$, as usual, the CR bundle of $M$. Then possibly the most basic invariant of $M$ is the complex distribution $\levinull\subseteq T^{1,0}M$ consisting of null vectors for any Levi form $\lambda$ of $M$: $Z_p\in \levinull_p$ if and only if $Z_p\in T^{1,0}_pM$ and $\lambda(Z_p,Z_p)=0$. We call $\levinull$ the Levi distribution of $M$. As is well-known, its definition is independent of the choice of Levi form $\lambda$. We can now give our main definition.

\begin{dfn}\label{levicore_dfn_intro}
The \textbf{Levi core} of $M$ is the core $\mathfrak{C}(\levinull)$ of the Levi distribution $\levinull$.
\end{dfn}

With a harmless abuse of language, we will also refer to $\mathfrak{C}(\levinull)$ as the Levi core of the domain $\Omega$. A couple of comments may help to clarify this definition.
\begin{enumerate}
    \item The pseudoconvex CR manifold $M$ (equivalently, the domain $\Omega$ of which $M$ is the boundary) is strongly pseudoconvex if and only if its Levi null distribution is trivial, i.e., $\levinull_p=0$ for every $p\in M$. Since the Levi core $\mathfrak{C}(\levinull)$ is typically much smaller than the Levi null distribution, one can view the category of smooth pseudoconvex domains with trivial Levi core as a wide class of "nondegenerate" pseudoconvex domains. In fact, in Section \ref{Levi_core_sec} we prove that weakly regular domains in the sense of Catlin \cite{Catlin_global} have trivial Levi core, and hence the same is true for pseudoconvex domains of finite-type in the sense of D'Angelo and pseudoconvex domains whose set of weakly pseudoconvex boundary points is a totally real submanifold.
    \item At the opposite end of the spectrum of smooth pseudoconvex domains, we have the most degenerate of all, namely those containing (positive dimensional) complex submanifolds $N$ in their boundary $M$. It is easy to see (Proposition \ref{horizontal_lem}) that in this case $T^{1,0}N\subseteq \mathfrak{C}(\levinull)$, and hence the Levi core of $M$ is nontrivial. In Section \ref{CRexamples_sec} we compute the Levi core for a family of three-dimensional pseudoconvex CR manifolds including the boundaries of generalized worm domains, showing how the Levi core captures the complex structure in the boundary, while forgetting additional "benign" weakly pseudoconvex points.
    \item In Section \ref{local_max_sec}, we show that any local maximum set $K\subseteq M$ in the sense of Slodkowski \cite{slodkowski_locmaxprop} is contained in the support of the Levi core, thus relating it to the \emph{weak Jensen boundary} of $\Omega$ (see  \cite{Ohsawa_Sibony_boundedpsh} for the definition and related results).
    \item Finally, let us conclude by mentioning the issue of Levi-flatness, which is a global version of the "most degenerate" case discussed in point (2). In presence of a Levi-flat open set in $M$, our construction does not yield any additional information, as the Levi core is supported on the whole open set (again by Proposition \ref{horizontal_lem}).
\end{enumerate}

\medskip

The Levi core interacts nicely with the two global invariants cited at the beginning of this introduction. We now turn our attention to them.

\medskip

The first is the \textbf{Diederich--Fornaess index} $\DF(\Omega)$, defined as the supremum of all exponents $\delta\in (0,1]$ with the property that $-(-r)^\delta$ is plurisubharmonic for at least a defining function $r$ of $\Omega$ (a notion originating from \cite{Diederich_Fornaess_index}). Among other reasons of interest, the Diederich--Fornaess index allows to formulate a sufficient condition for exact regularity of the $\dbar$-Neumann problem on a smooth bounded pseudoconvex domain $\Omega$ (see Straube's book \cite{Straube_book} for background on the $\dbar$-Neumann problem). Slightly more precisely, by a theorem of Kohn \cite{Kohn_quantitative}, under the assumption that $\DF(\Omega)=1$ and a quantitative control on a sequence of defining functions $r_k$ achieving the supremum in the definition of $\DF(\Omega)$, the $\dbar$-Neumann problem on $\Omega$ is exactly regular (see Section \ref{df_sec} for a precise statement and a generalization due to Harrington). Thus, it is of great interest to be able to compute or estimate the Diederich--Fornaess index of a domain, and in particular to decide whether $\DF(\Omega)=1$. In the last couple of years, some progress has been made on this quite difficult problem by Liu, Yum, and Adachi. In order to review these developments, we first need to say a few words about the second global invariant mentioned above, which we call (with a nonstandard terminology) the D'Angelo class.

\medskip

The \textbf{D'Angelo class} is a collection $\mathcal{A}_M$ of real smooth one-forms, called D'Angelo forms, that can be naturally attached to any pseudoconvex CR manifold of hypersurface type $M$, and that behaves as a sort of cohomology class "when restricted to the Levi distribution $\levinull$", in the sense that:\begin{itemize}
    \item[i)] if $\alpha \in \mathcal{A}_M$, then for any other $\alpha'\in \mathcal{A}_M$ there exists $f\in C^\infty(M,\R)$ such that $\alpha'_{|\levinull}=\alpha_{|\levinull}+df_{|\levinull}$.
    \item[ii)] if $\alpha \in \mathcal{A}_M$, then the two-form $d\alpha$ vanishes when restricted to $\levinull$.
\end{itemize}
The notion of D'Angelo class implicitly originated in papers of D'Angelo \cite{DAngelo_finite_type_real_hypersurfaces, DAngelo_iterated}, but it was in work of Boas and Straube \cite{Boas_Straube_derham} that its importance for the $\dbar$-Neumann problem was recognized. Boas and Straube proved that if $\Omega$ is a smooth bounded pseudoconvex domain and the points of infinite type of $M=b\Omega$ are all contained in a submanifold $N\subset M$ whose real tangent bundle is contained in $\Re(\levinull)$ (the real part of the Levi distribution), then a sufficient condition for exact regularity of the $\dbar$-Neumann problem on $\Omega$ is that the restriction of the D'Angelo class $\mathcal{A}_M$ to $N$ is trivial as an element of the first de Rham cohomology group $H^1_{\mathrm{dR}}(N,\R)$. This restriction is in fact a genuine de Rham cohomology class on $N$, thanks to properties i) and ii) above. A few other papers where the D'Angelo class plays a more or less explicit r\^{o}le appeared later, e.g., \cite{Straube_Sucheston_Exactness, Straube_Sucheston_Levi_foliations, Straube_sufficient, Forstneric_LaurentT_stein_compacts, Mongodi_Tomassini_semihol, Mongodi_Tomassini_onecomplete}.

\medskip

We can now discuss the recent work of Liu, Yum, and Adachi, establishing an interesting connection between the Diederich--Fornaess index of a smooth bounded pseudoconvex domain $\Omega$ and the D'Angelo class of its boundary $M$. Liu \cite{Liu_index_I} expressed $\DF(\Omega)$ as the optimal constant in a rather complicated differential inequality on $M$, and he was also able to determine the explicit value of this constant when $\Omega$ is a Diederich--Fornaess worm domain, thus succeeding in the exact computation of the $\DF$ index of a domain in $\C^n$ for which this quantity is strictly less than one, a result with remarkably no precedent in the literature (see \cite{fu_shaw_df, adachi_brinkschulte_global} for related results in the setting of complex manifolds). Next, Yum \cite{Yum_invariance} recognized that Liu's differential inequality could be neatly reformulated in terms of D'Angelo forms of $M$, and exploited this fact to prove the CR invariance of the $\DF$ index and of a "dual" Steinness index, introduced in a previous paper \cite{Yum_Steinness}. Finally, Adachi and Yum \cite{Adachi_Yum} generalized these results, replacing the ambient space $\C^n$ with an arbitrary complex manifold, and the extrinsic computations of \cite{Liu_index_I} and \cite{Yum_invariance}, which in particular relied on the flatness of the Euclidean metric, with intrinsic ones.

\medskip

Let us discuss Adachi--Yum result in more detail. If $\alpha\in \mathcal{A}_M$ is a D'Angelo form (for simplicity, in this introductory discussion we still assume that $M=b\Omega$, where $\Omega$ is a smooth bounded pseudoconvex domain in $\C^n$), we define the norm-like quantity \begin{equation}\label{norm_intro}
\norm(\alpha):=\inf\{t>0\colon \ \alpha_{1,0}\wedge\alpha_{0,1}<t\dbar\alpha_{1,0}\quad \text{on }\levinull\}.
\end{equation}
Here one may compute the $(1,0)$- and $(0,1)$-components of $\alpha$, and $\dbar\alpha_{1,0}$, by first extending $\alpha$ to an open neighborhood of $M$. It turns out that both $\alpha_{1,0}\wedge\alpha_{0,1}$ and $\dbar\alpha_{1,0}$ restrict to well-defined Hermitian forms on the Levi null distribution $\levinull$ (independently of the extension used to define them), and therefore the inequality in the sense of quadratic forms \[\alpha_{1,0}\wedge\alpha_{0,1}<t\dbar\alpha_{1,0}\] appearing in \eqref{norm_intro} makes sense. Next, we define \[
\norm:=\inf_{\alpha\in \mathcal{A}_M}\norm(\alpha),
\]
which should be thought of as a quantitative measure of the size of the D'Angelo class of $M$. Adachi--Yum main theorem (in the case of domains in $\C^n$) may be rephrased as follows.

\begin{thm}[Theorem 1.1. of \cite{Yum_invariance}; Theorem 2 of \cite{Adachi_Yum}]\label{adachi_yum_thm_intro} \begin{equation}\label{adachi_yum_formula}
\DF(\Omega)=\frac{1}{1+\norm}.
\end{equation}
\end{thm}

Since $\norm\in [0,+\infty)$ ($\norm=+\infty$ can be shown to be impossible if the ambient space is $\C^n$), one recovers the classical fact that $\DF(\Omega)>0$ (originally proved by Diederich and Fornaess \cite{Diederich_Fornaess_index}), and one can also see that $\DF(\Omega)=1$ if and only if $\norm=0$, a triviality condition for the D'Angelo class.

\medskip

Formula \eqref{adachi_yum_formula} elegantly relates the Diederich--Fornaess index and the D'Angelo class. We now bring the Levi core into the picture, showing how it further clarifies it. We start by generalizing the norm-like object \eqref{norm_intro}, embedding it into a family of similar gadgets: \begin{equation}\label{norm_D}
\norm(\alpha;\mathcal{D}):=\inf\{t>0\colon \ \alpha_{1,0}\wedge\alpha_{0,1}<t\dbar\alpha_{1,0}\quad \text{on }\mathcal{D}\},
\end{equation}
where the Levi null distribution $\levinull$ has been replaced by an arbitrary sub-distribution $\mathcal{D}\subseteq \levinull$. It is clear that the definition makes sense, and that $\norm(\alpha;\mathcal{D})\leq \norm(\alpha;\levinull)=\norm(\alpha)$ for every D'Angelo form $\alpha$. Next, we set \[
\norm(\mathcal{D}):=\inf_{\alpha\in \mathcal{A}_M}\norm(\alpha;\mathcal{D}).
\]
Our main result is the following "reduction to the core" theorem.

\begin{thm}\label{reduction_thm_intro} We have the identity:
\begin{equation}\label{reduction_intro}
\norm(\levinull)=\norm(\mathfrak{C}(\levinull)).
\end{equation}
\end{thm}

Combining Theorem \ref{adachi_yum_thm_intro} and Theorem \ref{reduction_thm_intro}, one can deduce a number of corollaries:
\begin{enumerate}
    \item[a)] If $\Omega$ has trivial Levi core, then $\norm=\norm(\levinull)=0$ and thus $\DF(\Omega)=1$. In particular, any weakly regular domain has $\DF$ index one. This includes the case of domains whose set of weakly pseudoconvex boundary points is a totally real submanifold, cf. Theorem 4.5 of \cite{liu_index_II}, in turn generalizing a result of Krantz, Liu, and Peloso \cite{krantz_liu_peloso_index}.

    \item[b)] A more precise quantitative refinement of Theorem \ref{reduction_thm_intro} (Theorem \ref{reduction_thm}) allows us to conclude that if the Levi core of $\Omega$ is trivial, then Kohn's theorem \cite{Kohn_quantitative} can be applied to conclude that the $\dbar$-Neumann problem is exactly regular (Theorem \ref{exact_reg_thm}). 
    \item[c)] If $M=b\Omega$ contains a complex submanifold $N$, then $\mathfrak{C}(\levinull)\supseteq T^{1,0}N$, as remarked after Definition \ref{levicore_dfn_intro}, and we have the trivial lower bound $\norm(\mathfrak{C}(\levinull))\geq \norm(T^{1,0}N)$. In Section \ref{de_rham_sec}, we prove that $\norm(T^{1,0}N)>0$ if and only if the restriction of the D'Angelo class to $N$ defines a nonzero element of $H^1_{\mathrm{dR}}(N, \R)$ (Theorem \ref{norm_thm}). Hence, the Diederich--Fornaess index is strictly less than one whenever this happens. Notice that, while this conclusion already follows from Theorem \ref{adachi_yum_thm_intro} and Theorem \ref{norm_thm}, by Theorem \ref{reduction_thm_intro} one has the equality $\norm=\norm(T^{1,0}M)$ whenever $\mathfrak{C}(\levinull)=T^{1,0}N$ (cf. the examples of Section \ref{CRexamples_sec}).
\end{enumerate}

\medskip

The paper is organized as follows: \begin{itemize}
    \item In Section \ref{core_sec} we define and prove the existence of the core of a distribution of subspaces $\mathcal{D}$ on a real smooth manifold, under the sole assumption that $\mathcal{D}$ is closed as a subset of the tangent bundle.
    \item In Section \ref{Levi_core_sec} we introduce the Levi core in the context of abstract pseudoconvex CR manifolds of hypersurface type and prove a number of its properties.
    \item Section \ref{d'angelo_sec} starts with a review of the basic theory of the D'Angelo class, again in the context of abstract pseudoconvex CR manifolds of hypersurface type. It continues with the definition of a norm-like function on the first de Rham cohomology of a complex manifold, and the proof of its nondegeneracy (Theorem \ref{norm_thm}). This definition is then used as a motivation for introducing the quantites $\norm(\mathcal{D})$ (see \eqref{norm_D} above), and certain quantitative refinements $\norm_K(\mathcal{D})$ of these (where $K<+\infty$ is a parameter).
    \item In Section \ref{reduction_sec} our main "reduction to the core" theorem is proved.
    \item Finally, Section \ref{df_sec} discusses a refinement of Adachi--Yum theorem involving the norms $\norm_K(\levinull)$, that allows in particular to deduce, via Kohn's theorem \cite{Kohn_quantitative}, the exact regularity of the $\dbar$-Neumann problem on smooth bounded pseudoconvex domains with trivial Levi core.
    \item In the Appendix, we give a quantitative estimate for the norm-like functions introduced in Section \ref{d'angelo_sec}, in the case where the complex manifold is a pseudoconvex domain in $\C^n$ satisfying a weak regularity property. 
\end{itemize}

\subsection{Acknowledgment} This project was begun while the first-named author was a Marie Sk\l odowska-Curie Research Fellow at the University of Birmingham. He gratefully
acknowledges the support of the European Commission via the Marie Sk\l odowska-Curie Individual Fellowship "Harmonic Analysis on Real Hypersurfaces in Complex Space" (ID 841094).

\section{The core of a distribution of subspaces}\label{core_sec}

Let us begin with a remark on terminology. From now on, we employ the term "distribution", in place of the cumbersome "distribution of subspaces", for the objects of Definition \ref{distribution} below. This terminology conflicts with that of the theory of distributions, or "generalized functions" (and the expression "support of a distribution" makes the conflict even worse), but since generalized functions will play no r\^{o}le in what follows no confusion should ensue. We will also write $\C V$ for the complexification $\C\otimes V$ of a real vector space $V$.

\subsection{Distributions and closed distributions}\label{distribution_sec}

\begin{dfn}[Distribution]\label{distribution}
	A \textbf{real distribution} on a real smooth manifold $M$ is a subset $\distr$ of the tangent bundle $TM$ such that the fiber $\distr_p:=\distr\cap T_pM$ is a vector subspace of $T_pM$ for every $p\in M$. A \textbf{complex distribution} on $M$ is a subset $\distr$ of the complexified tangent bundle $\C TM$ such that $\distr_p:=\distr\cap \C T_pM$ is a complex vector subspace of $\C T_pM$ for every $p\in M$.
	
	If $\distr$ is a real distribution, then the complexification of $\distr$ is the complex distribution $\C \distr$ defined by \begin{equation*}
	\C \distr_p:=\{X_p+iY_p\colon X_p, Y_p\in \distr_p\}.
	\end{equation*}
	
	If $\distr$ is a complex distribution, then its real part is the real distribution $\Re(\distr)$ defined by \begin{equation*}
	\Re(\distr)_p:=\{Z_p+\overline{Z}_p\colon\  Z_p\in \distr_p\}.
	\end{equation*}
	
	The \textbf{support} $S_{\distr}$ of a (real or complex) distribution $\distr$ is the set of points $p\in M$ such that $\distr_p\neq \{0\}$.
\end{dfn}

As a first elementary example, we observe that distributions on a one-dimensional manifold are nothing but subsets.

\begin{ex}[Distributions on a one-dim. manifold]\label{one_dim_distribution}
If $M$ is a one-dimensional manifold and $\mathcal{D}$ is a real (resp. complex) distribution, for each $p\in M$ we have either $\distr_p=\{0\}$ or $\distr_p=T_pM$ (resp. $\distr_p=\C T_pM$). Therefore, we can recover $\distr$ from its support $S_{\distr}\subseteq M$. Viceversa, it is clear that for every subset $S\subseteq M$ there exists a unique real (resp. complex) distribution $\distr$ such that $S_\distr=S$.
\end{ex}

We will be mostly interested in \textbf{closed} distributions, that is, distributions $\mathcal{D}$ that are closed subsets of $TM$, or $\C TM$.

The reader may verify that the complexification of a closed real distribution is closed. Another very basic property of closed distributions is contained in the next proposition, which implies in particular that the supports of closed distributions are closed.

\begin{prp}\label{upper_semicontinuous_prp}
	If $\distr$ is a closed real (resp. complex) distribution, then the dimension $\dim_\R\distr_p$ (resp. $\dim_\C\distr_p$) is an upper semicontinuous function of the point $p$.
\end{prp}

\begin{proof} We write the proof in the real case, the complex case being a trivial variant. We argue by contradiction, assuming that $\dim_\R \distr_p=k$, while $\dim_\R \distr_{p_n}\geq k+1$ for a sequence of points $p_n$ converging to $p$. Fixing an arbitrary Riemannian metric $g$ on $M$, we may find orthonormal sets of vectors $\{v_1^{(n)}, \ldots, v_{k+1}^{(n)}\}\subseteq \mathcal{D}_{p_n}$. By a simple compactness argument, we may assume without loss of generality that $v_j^{(n)}$ converges to $v_j\in T_pM$ for every $j=1,\ldots, k+1$. By the closure assumption, $v_1,\ldots, v_{k+1}$ is an orthonormal subset of $\mathcal{D}_p$, contradicting the fact that $\dim_\R \distr_p=k$.
\end{proof}

\begin{rmk}
Another interesting class of distributions is that of \textbf{smooth distributions} (see, e.g., \cite{Lavau}), i.e., distributions such that for every $p\in M$ and every $v\in \distr_p$ there exists a local smooth vector field $X$ on a neighborhood $U$ of $p$ such that $X_p=v$ and $X_q\in \distr_q$ for every $q\in U$. One may easily prove that, for smooth distributions, $\dim\mathcal{D}_p$ is a lower semicontinuous function of $p\in M$. It follows that a distribution is simultaneously smooth and closed if and only if it is a subbundle of the tangent bundle.

\end{rmk}

It is easy to see that a one-dimensional distribution $\distr$ as in Example \ref{one_dim_distribution} is closed if and only if its support $S_{\distr}$ is a closed subset of $M$.

Less obvious examples of closed distributions arise as null-distributions of continuous quadratic forms.

\begin{ex}[Null-distribution of a continuous quadratic form]\label{ex_kernel} Let $E$ be a subbundle of $TM$ and let
\[B_p:E_p\times E_p\to\R\]

be a bilinear symmetric form depending continuously on $p$. Its null-space defines a closed distribution on $M$:
$$\ker B=\{(p, X_p)\in E\ :\ B_p(X_p, Y_p)=0\quad \forall Y_p\in\ E_p\}\;.$$
To see that this distribution is closed, observe that $$\ker B=\bigcap_Y\{(p, X_p)\in E:\ B_p(X_p, Y_p)=0\},\;$$ where $Y$ ranges over global continuous sections of $E$.

Analogously, one may define the null-distribution $\ker H$ of a Hermitian form $H_p: E_p\times \overline{E_p}\to\R$ depending continuously on $p$, where now $E_p$ is a complex subbundle of $\C T_pM$. The resulting complex distribution $\ker H$ is also closed.
\end{ex}

Another general construction yielding closed distributions is provided by the following notion of Zariski tangent space in the smooth category.

\begin{dfn}[\textbf{Tangent distribution} to a subset]\label{tangent_distribution}
	Let $A$ be an arbitrary subset of a real smooth manifold $M$. The (real) tangent distribution $TA$ to $A$, whose fibers we denote by $T_pA$, is defined as follows: \begin{equation*}
	X_p\in T_pA\ \Longleftrightarrow\ X_pf=0\quad\forall f\in C^\infty(M) \colon\ f_{|A}\equiv 0.
	\end{equation*}
	The complex tangent distribution to $A$ is the complexified distribution $\C TA$.
\end{dfn}

Observe that $TA=T\overline{A}$ for any set $A$, and that tangent distributions are local: if $U\subseteq M$ is open, then $TA$ and the tangent distribution to $A\cap U$ as a subset of $U$ have the same fiber at every $p\in U$.

\begin{rmk}If $A$ is a closed subset of $M$ and $\mathcal{C}^\infty_A$ is the sheaf of germs of restrictions to $A$ of smooth functions on $M$, the pair $(A,\mathcal{C}^\infty_A)$ is a (reduced) differentiable space in the terminology of \cite{Navarro_Gonzalez} (see \cite[Theorem 3.23]{Navarro_Gonzalez}). The space $T_pA$ of Definition \ref{tangent_distribution} coincides with the tangent space defined in \cite[Sections 5.4 and 5.5]{Navarro_Gonzalez}.
\end{rmk}

The next straightforward proposition lists a few basic properties of tangent distributions.

\begin{prp}\label{tangent_prp}\
\begin{itemize}
    \item[a)] The tangent distribution $TA$ to any set $A$ is a closed distribution, whose support is the set of limit points of $A$.
    \item[b)] If $A$ is an embedded submanifold of $M$, then $T_pA$ is the ordinary tangent space to $A$ at every point $p\in A$ (thought of as a subspace of $T_pM$).
    \item[c)] The dimension of $T_pA$ is the minimal $k\in \N$ such that $A$ is locally contained, near $p$, in a real embedded $k$-dimensional submanifold $V$ of $M$. In this case, $T_pA=T_pV$.
\end{itemize}
\end{prp}

\begin{proof}
Notice that, by definition,
\begin{equation}\label{TS}TA=\bigcap_f \ker df\end{equation}
where the intersection is taken over all the functions $f\in C^\infty(M)$ such that $f_{\vert A}=0$.
Therefore $TA$ is closed as a subset of $TM$.

If $p\in M$ is not a limit point of $A$, we have a neighborhood $U$ of $p$ such that $U\cap A\subseteq\{p\}$. Let $x_1,\ldots, x_n$ be local coordinates centered at $p$ and defined on such a neighborhood. Since every $x_j$ is a function vanishing on $A$, one easily deduces that $T_pA=\{0\}$. On the other hand, if $T_pA=\{0\}$, by \eqref{TS} and elementary linear algebra we find $f_1,\ldots, f_n\in C^\infty(M)$ vanishing on $A$ and such that $$\ker df_1(p)\cap\ldots\cap\ker df_n(p)=\{0\}.$$
Therefore these functions give a set of local coordinates on a neighborhood $U$ of $p$, and their common zero set in $U$ has to be $\{p\}$, which cannot be a limit point of $A$. This proves a).

Let us prove b). If $A$ is a submanifold of $M$, for $p\in A$ we have coordinates $x_1,\ldots, x_n$ on a neighborhood $U$ of $p$ such that $$A\cap U=\{q\in U\ :\ x_1(q)=\ldots=x_k(q)=0\}\;.$$ We may assume that $x_1,\ldots, x_n$ are global smooth functions on $M$ by multiplying them with a cut-off function and restricting $U$, as usual. The local coordinates give a frame $X_1,\ldots, X_n$ for $TM$ over $U$ such that $X_jx_k=\delta_{jk}$ (Kronecker delta) and $T_pA= \mathrm{Span}\{X_{k+1, p},\ldots,X_{n, p}\}$. Moreover, $f\in C^\infty(M)$ is such that $f_{\vert A}=0$ if and only if we can write $$f=x_1g_1+\ldots+x_kg_k\qquad \text{on }U$$ for some functions $g_1,\ldots, g_k\in C^\infty(U)$. From this, one readily deduces that $T_pA$ is spanned by $X_{k+1}, \ldots, X_n$, i.e., it coincides with the ordinary tangent space to $A$ at $p$.

We are left with $c)$. If $T_pA$ has dimension $k$, by \eqref{TS} we can find $n-k$ functions $f_{k+1},\ldots, f_n\in C^\infty(M)$ vanishing on $A$ and such that
$$\ker df_1(p)\cap\ldots\cap\ker df_n(p)=T_pA\;.$$
Therefore, there is a neighborhood $U$ of $p$ such that
$$V=\{q\in U\ :\ f_{k+1}(q)=\ldots=f_n(q)=0\}$$
is a $k$-dimensional embedded submanifold which  contains $A\cap U$, and such that $T_pV=T_pA$ (here one uses part b)).
On the other hand, if $S$ is locally contained in an embedded manifold $W$, then $T_pS\subseteq T_pW$, as any function that vanishes on $W$ near $p$ also vanishes on $S$ near $p$. This shows the minimality of $k$ and concludes the proof.
\end{proof}

\subsection{Derived distribution and core of a distribution}\label{derived_sec}

On a one-dimensional manifold $M$ we can identify a distribution with its support (Example \ref{one_dim_distribution}). Hence, by Proposition \ref{tangent_prp}, the tangent distribution to a subset $S$ of $M$ can be identified with the set of limit points of $S$, i.e., the so-called derived set of $S$. This motivates the terminology in the following definition.

\begin{dfn}[\textbf{Derived distribution}]\label{derived_distribution}
	Let $\distr\subseteq TM$ be a real distribution. The distribution \begin{equation*}
	\distr':=\distr\cap TS_{\distr}
	\end{equation*} will be called the derived distribution of $\distr$. Analogously, if $\distr\subseteq \C TM$ is a complex distribution on $M$, its derived distribution is defined as \begin{equation*}
	\distr':=\distr\cap \C TS_{\distr}.
	\end{equation*}
	
	In analogy with the corresponding notion for sets, a distribution equal to its derived distribution will be called \textbf{perfect}.
\end{dfn}

Notice that while the definition makes perfect sense for general distributions, we will use it only for closed ones. We collect a few basic properties of derived distributions in the next proposition.

\begin{prp}\label{derived_prp}\
\begin{itemize}
    \item[a)] The derived distribution of a closed (real or complex) distribution is closed.
    \item[b)] If $S\subseteq M$ is perfect (that is, it coincides with the set of its limit points), then its tangent distributions $TS$ and $\C TS$ are perfect.
    \item[c)] Complexification and the operation of taking the derived distribution commute: $\C \mathcal{D}'=(\C\mathcal{D})'$ for every real distribution $\mathcal{D}$.
    \item[d)] If $\mathcal{D}$ and $\mathcal{E}$ are two (real or complex) distributions such that $\mathcal{D}\subseteq \mathcal{E}$, then $\mathcal{D}'\subseteq\mathcal{E}'$.
\end{itemize}
\end{prp}

\begin{proof}
Parts a) and b) follow from part a) of Proposition \ref{tangent_prp}. Part c) is trivial. Part d) is an immediate consequence of the implications $\mathcal{D}\subseteq\mathcal{E}\Longrightarrow S_{\mathcal{D}}\subseteq S_{\mathcal{E}}\Longrightarrow (\C)TS_{\mathcal{D}}\subseteq (\C)TS_{\mathcal{E}}$ plus part c).
\end{proof}

\begin{rmk} In one dimension, the derived distribution of the distribution supported on the closed set $S$ (see Example \ref{one_dim_distribution}) is the distribution supported on its derived set. The same may not be true if $S$ is not closed, because not every limit point of $S$ is an element of $S$ in general.
\end{rmk}

It is natural to iterate the operation of "taking the derived distribution". Let then $\distr$ be a (real or complex) distribution and let $\alpha$ be an ordinal. We define $\distr^{(\alpha)}$ by transfinite recursion (see \cite{Ciesielski}, Theorem 4.3.1):\begin{enumerate}
	\item if $\alpha=0$, we set $\distr^{(0)}:=\distr$;
	\item if $\alpha+1$ is a successor ordinal, we set $\distr^{(\alpha+1)}:=(\distr^{(\alpha)})'$;
	\item if $\alpha$ is a limit ordinal, we set $\distr^{(\alpha)}:=\bigcap_{\beta<\alpha}\distr^{(\beta)}$.
\end{enumerate}
If $\distr$ is closed, then every $\distr^{(\alpha)}$ is closed (this is easily proved by transfinite induction). We have the following Cantor--Bendixson-type theorem (cf. \cite[Thm. 6.2.4]{Ciesielski}).

\begin{thm}\label{Cantor_Bendixson}
	If $\distr$ is a closed (real or complex) distribution, then there exists a countable ordinal $\alpha$ such that $\distr^{(\alpha)}$ is perfect.
\end{thm}

\begin{proof}
As it is customary, we denote by $\omega_1$ the smallest uncountable ordinal. Recall that $\omega_1=\{\alpha\colon\ \alpha<\omega_1\}$. The recursion above gives a function \begin{eqnarray*}
		\omega_1&\rightarrow& \{\text{closed distributions on }M\}\\
		\alpha&\mapsto& \distr^{(\alpha)}.
	\end{eqnarray*} satisfying the monotonicity property \begin{equation*}
	\alpha_1<\alpha_2 \Longrightarrow \distr^{(\alpha_1)}\supseteq \distr^{(\alpha_2)}.
	\end{equation*}
Assume by contradiction that no $\distr^{(\alpha)}$, with $\alpha<\omega_1$, is perfect. Then the mapping above is injective, and this contradicts the fact that if $\{C_\alpha\colon\ \alpha<\beta\}$ is a strictly decreasing family of closed subsets of $TM$, then the ordinal $\beta$ is at most countable. This is part (iii) of Theorem 6.2.1 of \cite{Ciesielski} (the result is stated there for $\R^n$ in place of $TM$, but the generalization is straightforward).
\end{proof}

Thanks to Theorem \ref{Cantor_Bendixson}, we can give the main definition of this section.

\begin{dfn}[\textbf{Core of a distribution}]\label{core_dfn}
Let $\distr$ be a (real or complex) closed distribution. Then $\distr^{(\omega_1)}$ will be called the core of $\distr$, and denoted $\core(\distr)$. Equivalently, $\core(\distr)=\distr^{(\alpha)}$, where $\alpha$ is the minimal ordinal such that $(\distr^{(\alpha)})'=\distr^{(\alpha)}$.
\end{dfn}

In view of the remark after Proposition \ref{derived_prp}, it is clear that there are plenty of examples of closed distributions $\mathcal{D}$ for which $\mathfrak{C}(\mathcal{D})\subsetneq \mathcal{D}^{(k)}$ for every finite $k$. Nevertheless, in the examples considered in this paper, it is always the case that $\mathfrak{C}(\mathcal{D})=\mathcal{D}^{(k)}$ for some finite $k$. 

It is worth looking explicitly at a simple two-dimensional example where the iteration operation stabilizes after more than one step.

\begin{ex} Let $M=\R^2$ and consider the quadratic form on $T\R^2$ defined by $B_{(x,y)}=xdx\otimes dx+ydy\otimes dy$. The closed distribution $\mathcal{D}:=\ker B$ (cf. Example \ref{ex_kernel}) is easily seen to be:\begin{eqnarray*}
\mathcal{D}_{(x,y)}=
\begin{cases}
\R\partial_x\qquad &x=0,\ y\neq 0\\
\R\partial_y\qquad &x\neq0,\ y= 0\\
\R\partial_x+\R\partial_y\qquad &x=0,\ y= 0\\
\{0\}\qquad &x\neq 0,\ y\neq 0
\end{cases}
\end{eqnarray*}
Hence, $S_{\mathcal{D}}$ is the union of the coordinate axes and \begin{eqnarray*}
(TS_{\mathcal{D}})_{(x,y)}=
\begin{cases}
\R\partial_y\qquad &x=0,\ y\neq 0\\
\R\partial_x\qquad &x\neq0,\ y= 0\\
\R\partial_x+\R\partial_y\qquad &x=0,\ y= 0\\
\{0\}\qquad &x\neq 0,\ y\neq 0
\end{cases}
\end{eqnarray*}
Therefore,
\begin{eqnarray*}
\mathcal{D}'_{(x,y)}=
\begin{cases}
\R\partial_x+\R\partial_y\qquad &x=0,\ y= 0\\
\{0\}\qquad &\text{otherwise}
\end{cases}
\end{eqnarray*}
The support of the derived distribution $\mathcal{D}'$ is the origin, and one finally gets that $\mathcal{D}''$ and $\mathfrak{C}(\mathcal{D})$ are the zero distribution.
\end{ex}

The next elementary lemma may be of help in describing 
the core of a distribution.

\begin{lem}\label{ovvio_lem} Suppose that $\distr$ is a closed (real or complex) distribution on $M$ and that $\mathcal{E}$ is a perfect (real or complex) distribution such that $\mathcal{E}\subseteq\distr$. Then $\mathcal{E}\subseteq \core(\distr)$. \end{lem}

\begin{proof} It is enough to show, by transfinite induction, that $\mathcal{E}\subseteq\distr^{(\alpha)}$ for every ordinal $\alpha$. The case of a successor ordinal follows from part d) of Proposition \ref{derived_prp}. The rest is trivial.
\end{proof}

Notice that the closedness of $\distr$ is not explicitly used in the proof. However, it is needed to ensure the existence of the core.

\subsection{A compactness lemma}\label{compactness_sec}

A subset $\Omega\subseteq (\C)TM$ is said to be \textbf{conical} if $v\in \Omega$ implies $tv\in \Omega$ for every $t>0$.

\begin{prp}\label{compactness_prp} Let $M$ be compact and let $\distr$ be a closed real (resp. complex) distribution. Then any cover of $\distr\setminus\left(M\times \{0\}\right)$ by open conical subsets of the real (resp. complexified) tangent bundle has a finite subcover.
\end{prp}

\begin{proof}
Fix a metric $g$ on $M$. Then we define $\sphere(\distr)$ as the set of vectors of unit length in $\distr$. The statement boils down to the compactness of $\sphere(\distr)$, a consequence of the closure of $\distr$ and the compactness of $M$.
\end{proof}

\section{The Levi core of a CR manifold of hypersurface type}\label{Levi_core_sec}

In this section we apply the basic theory of Section \ref{core_sec} to CR manifolds of hypersurface type. In Section \ref{CR_sec} we review the elements of CR geometry we need (cf. \cite{dragomir_tomassini} for a deeper discussion) and we introduce the notion of Levi core. Then we look at a few examples (in Section \ref{CRexamples_sec}) and we specialize the discussion to real hypersurfaces of complex manifolds (in Section \ref{hypersurfaces_sec}). Finally, in Section \ref{local_max_sec}, we discuss the relation between Levi core and local maximum sets. 

\subsection{Notation and basic definitions}\label{CR_sec}

We consider an orientable, connected \textbf{CR manifold} $(M, T^{1,0}M)$ \textbf{of hypersurface type} (in an alternative common terminology, we require $(M, T^{1,0}M)$ to be "of type $(n,1)$"). This means that $M$ is a real smooth manifold of odd dimension $2n+1$ ($n\geq 1$), and that $T^{1,0}M$ is a complex subbundle of $\C TM$ of complex rank $n$ such that the following two properties hold:
\begin{enumerate}
    \item $T^{1,0}M\cap T^{0,1}M=0$, where $T^{0,1}M:=\overline{T^{1,0}M}$;
    \item $T^{1,0}M$ is formally integrable, that is, the commutator of any pair of smooth sections of $T^{1,0}M$ is also a section of $T^{1,0}M$.
\end{enumerate}

From now on, we omit the specification "of hypersurface type" and we refer to any $M$ equipped with a structure as above as a "CR manifold". 

We denote by $H(M)\subset TM$ the maximal complex distribution of the CR manifold, i.e.,
$$H(M)=\Re(T^{1,0}M)=\Re(T^{1,0}M\oplus T^{0,1}M)\;.$$
The real vector bundle $H(M)$ has real rank $2n$ and carries a complex structure $J_b$ defined by $J_b(Z+\overline{Z})=i(Z-\overline{Z})$ for every section $Z$ of $T^{1,0}M$. The real part $\Re:T^{1,0}M\rightarrow H(M)$ is an isomorphism of complex vector bundles with respect to this structure.

Consider now the real line bundle $E$ over $M$ defined by
$$E_p=\{\omega\in T^*_pM\ :\ \ker\omega\supseteq H_p(M)\}\;\qquad (p\in M),$$
where $H_p(M)$ is the fiber at $p$ of the maximal complex distribution. It is easy to see that $E$ is isomorphic to $TM/H(M)$. Since $M$ is orientable and $H(M)$ is oriented by the complex structure $J_b$, it follows that $E$ is orientable. Being defined over a connected manifold, it is globally trivial. We denote by $\Theta(M, T^{1,0}M)$ the set of global nowhere vanishing smooth sections of $E$, that is, $\theta\in \Theta(M,T^{1,0}M)$ if and only if $\theta$ is a real nowhere vanishing smooth one-form with the property that $\theta(Z)=0$ for every section $Z$ of $T^{1,0}M$. Elements of $\Theta(M,T^{1,0}M)$ are called \textbf{pseudo-Hermitian structures} on $(M,T^{1,0}M)$.

Given a pseudo-Hermitian structure $\theta$, the associated \textbf{Levi form} is the Hermitian form on $T^{1,0}M$ defined by
$$\lambda_\theta(Z,\overline{W}):=\dfrac{1}{2i}d\theta(Z,\overline{W})=\dfrac{i}{4}\theta([Z,\overline{W}])\;,$$ where $Z,W$ are smooth sections of $T^{1,0}M$. If $\theta_1,\ \theta_2\in\Theta(M,T^{1,0}M)$, then there exists a nowhere vanishing $g\in C^\infty(M,\R)$ such that $\theta_1=g\theta_2$, and we have \begin{equation}\label{eq_cambiotheta}\lambda_{\theta_1}=g\lambda_{\theta_2}\;.\end{equation}

We say that $(M, T^{1,0}M)$ is \textbf{pseudoconvex} if $\lambda_{\theta}$ is semidefinite for some (and hence for every) $\theta\in\Theta(M,T^{1,0}M)$. If $(M, T^{1,0}M)$ is pseudoconvex, we denote by $\Theta_+(M,T^{1,0}M)$ the set of pseudo-Hermitian structures such that the corresponding Levi form $\lambda_\theta$ is nonnegative definite.

From Equation \eqref{eq_cambiotheta}, it follows that the null-distribution (see Example \ref{ex_kernel}) of the Levi form $\lambda_\theta$ is independent of $\theta$: we call it the \textbf{Levi distribution} of the CR manifold and denote it by $\levinull$. As observed in Example \ref{ex_kernel}, $\levinull$ is a closed complex distribution on $M$. Explicitly, for $p\in M$ we have
$$\levinull_p=\{Z_p\in T^{1,0}_pM\ :\ \lambda_{\theta,p}(Z_p,\overline{W}_p)=0\quad \forall W_p\in T^{1,0}M\}\;.$$
If $(M, T^{1,0}M)$ is pseudoconvex, by elementary linear algebra one sees that $\levinull_p$ consists of those vectors $Z_p\in T_p^{1,0}M$ such that $\lambda_{\theta,p}(Z_p,\overline{Z}_p)=0$. We finally recall the main definition of the paper (anticipated in the introduction).

\begin{dfn} The \textbf{Levi core} of a CR manifold is the core $\mathfrak{C}(\mathcal{N})$ of its Levi distribution $\levinull$ (in the sense of Definition \ref{core_dfn}).
\end{dfn}

The next two propositions allow to control, from below and above respectively, the Levi core in various situations.

\begin{prp}\label{horizontal_lem}
Let $V\subset M$ be an embedded submanifold of positive dimension and such that $T_pV$ is a complex (i.e., $J_b$-invariant) subspace of $H_p(M)$ for every $p\in V$. Then \begin{equation}\label{lower_core}
\{(p,Z_p)\colon \ p\in V\text{ and } Z_p\in \C T_pV\cap T^{1,0}_pM\}\subseteq \core(\levinull).
\end{equation}
\end{prp}

As observed in the proof, the LHS of \eqref{lower_core} is nontrivial for every $p\in V$, and therefore the proposition implies in particular that $S_{\mathfrak{C}(\mathcal{N})}\supseteq V$.

\begin{proof} Denote by $\mathcal{E}$ the complex distribution on $M$ defined by the LHS of \eqref{lower_core} (and trivial fiber for $p\notin V$). Thanks to Lemma \ref{ovvio_lem}, our task boils down to proving the following two statements: \begin{itemize}
    \item[a)] $\mathcal{E}$ is a perfect distribution.
    \item[b)] $\mathcal{E}\subseteq \mathcal{N}$.
\end{itemize}

To prove statement a) it is clearly enough to check that $V$ is contained in the support of $\mathcal{E}$. Given $p\in V$ and a nonzero $X_p\in T_pV$ (it exists because $V$ is positive dimensional), we have $X_p-iJ_bX_p\in \C T_pV\cap T^{1,0}M$, because $T_pV$ is $J_b$-invariant. Thus, the fiber at $p$ of $\C TV\cap T^{1,0}M$ is nontrivial, as we wanted.

To verify the inclusion b), we pick $X_p\in T_pV$ ($p\in V$) and check that $Z_p:=X_p-iJ_bX_p\in \mathcal{N}_p$. Let $X$ be a local smooth section of $TM$ extending $X_p$ and such that $X_q\in T_qV$ for every $q\in V$ near $p$, and put $Z:=X-iJX$.
By the $J_b$-invariance assumption, the restriction of $Z$ to $V$ is a section of $\C TV$. Then $[Z,\overline{Z}]_q\in \C T_qV$ at every $q\in V$ near $p$. Because $\C T_pV\subseteq T^{1,0}_pM\oplus T^{0,1}_pM$, this implies in turn that $\theta([Z,\overline{Z}])$ vanishes on $V$ for every pseudo-Hermitian structure $\theta$. Since $M$ is pseudoconvex, this means that $Z_p\in \mathcal{N}_p$, completing the proof of b).\end{proof}

Before stating the second proposition, we need a couple of definitions, where $M$ is always a CR manifold of hypersurface type.

\begin{dfn}[\textbf{Zero holomorphic dimension}]
A subset $A\subseteq M$ has zero holomorphic dimension at $p$ if
$$\C T_pA\cap \levinull_p=\{0\}\;.$$
\end{dfn}
Notice that here $T_pA$ is the tangent space in the sense of Definition \ref{tangent_distribution}.

\begin{dfn}[\textbf{Weakly regular stratification}]\label{weakly_reg_dec_dfn}
A weakly regular stratification of $M$ is a finite collection of closed subsets $\Sigma_i\subseteq M$ ($i=0,\ldots, N$) such that:
\begin{enumerate}
    \item $\Sigma_N\subset \Sigma_{N-1}\subset\ldots\subset \Sigma_1\subset \Sigma_0=M$
    \item for each $i=0,\ldots, N-1$, the set $\Sigma_i\setminus \Sigma_{i+1}$ has zero holomorphic dimension at each of its points.
\end{enumerate}

We say that $M$ is \textbf{weakly regular} if there exists a weakly regular decomposition of $M$ with $\Sigma_N=\emptyset$.
\end{dfn}

Notice that, since the $\Sigma_i$'s are closed and tangent distributions are local objects, the assumption is equivalent to $\C T_p\Sigma_i\cap \levinull_p=\{0\}$ for every $p\in \Sigma_i\setminus \Sigma_{i+1}$ and every $i=0,\ldots, N-1$.

\begin{rmk}
The notion of zero holomorphic dimension originates from work of Kohn \cite{Kohn_subelliptic} (see also \cite[p. 40]{Catlin_global}). Definition \ref{weakly_reg_dec_dfn} generalizes the notion of weakly regular boundary in \cite{Catlin_global} (which in turn had a precursor in \cite{Diederich_Fornaess_analytic}). In Catlin's definition, $M$ is the boundary of a smooth bounded domain, and $\Sigma_N$ is empty. Our definition is phrased slightly differently in terms of smooth Zariski tangent spaces, but it is an immediate consequence of part c) of Proposition \ref{tangent_prp} that the two definitions coincide, modulo these additional assumptions.
\end{rmk}

\begin{prp}\label{weaklydecomp_prp}If $\{\Sigma_j\}_{j=0,\ldots, N}$ is a weakly regular stratification of $M$, then $S_{\core(\levinull)}\subseteq \Sigma_N$.\end{prp}

\begin{proof}
Without loss of generality, we may assume that the Levi core is nontrivial. Let $j$ be the largest index for which $S_{\mathfrak{C}(\mathcal{N})}\subseteq \Sigma_j$, and assume by contradiction that $j<N$. By the stability of the Levi core, \[
\mathfrak{C}(\mathcal{N}) = \mathfrak{C}(\mathcal{N})\cap \C TS_{\mathfrak{C}(\mathcal{N})}\subseteq \mathcal{N}\cap \C T\Sigma_j.
\]
If $p\in \Sigma_j\setminus \Sigma_{j+1}$ is in the support of the Levi core, then $\mathfrak{C}(\mathcal{N})_p\subseteq \mathcal{N}_p\cap \C T_p\Sigma_j=\{0\}$, contradicting the maximality of $j$.
\end{proof}

Applying Proposition \ref{weaklydecomp_prp} to the case of a weakly regular CR manifold, we obtain the following.

\begin{cor}
If $M$ is weakly regular, then its Levi core is trivial.
\end{cor}

Finally, we define $\levinull^\R:=\Re(\levinull)$, which is a real closed distribution on $M$. Since taking the real part of a vector field gives an isomorphism of complex vector bundles between $T^{1,0}M$ and $H(M)$, $\levinull^\R_p$ is a complex subspace of $H(M)_p$ for every $p\in M$. For the same reason, $\Re(\core(\levinull))_p$ is, at every $p\in M$, a complex subspace of $H(M)_p$. However, $\core(\levinull^\R)_p$ need not be. 
We will be mostly using the distribution $\levinull$ instead of $\levinull^\R$, even if the latter may have a more immediate geometric meaning, particularly when Levi-flat foliations or complex analytic varieties are present.

\subsection{A class of $3$-dimensional CR structures}\label{CRexamples_sec}

In this section, we describe the Levi core for certain CR structures on the total space of (real) line and circle bundles on Riemann surfaces. A similar construction was examined in  \cite{cordaro}.

\medskip

Let $Y$ be a Riemann surface. We consider a fiber bundle $\pi:M\to Y$ whose fibers are either $F=\R$ or $F=\mathbb{S}^1$.
Let $t$ be a global coordinate on $F$ and $z$ a local coordinate on $Y$. We define a CR structure on $M$ as follows.

Consider a $(1,0)$-form $\omega$ on $Y$ and a smooth complex-valued function $g\in C^\infty(F)$; we define $T^{1,0}M$ as the common kernel of the two $1$-forms:
$$dt-g(t)\omega\quad\text{and}\quad d\overline{z}\;.$$
If we write locally $\omega=u(z)dz$, then
$$L=\partial_z+g(t)u(z)\partial_t$$
is a local generator for $T^{1,0}M$; hence,
$$\theta=dt-2\Re(g(t)\omega)$$
is a pseudo-Hermitian structure for $(M,T^{1,0}M)$. The Levi form of $\theta$ is locally given by
$$\lambda_{\theta}(L,\overline{L})=\frac{i}{4}\theta([L,\overline{L}])=
\frac{1}{2}\Im(gu_{\bar{z}})+\frac{|u|^2}{2}\Im(\bar{g}g_t)\;.$$

In order to have a pseudoconvex CR manifold, we want $\lambda_{\theta}$ to be semidefinite, hence we want to control the sign of the previous quantity, in terms of $g$ and $\omega$. In order to do that, we need a couple of observations. If $Y$ is compact, there are a holomorphic $1$-form $\eta$ and a smooth function $H$ such that $\omega=\eta+\partial H$ (it follows, for instance, by taking the conjugate of \cite[Theorem 19.9]{Forster_Rsurf}); therefore, locally,
$$u(z)=f(z)+H_{z}(z)$$
with $f(z)$ holomorphic. If instead $Y$ is open, then it is Stein and $H^{p,q}(Y)=0$ for $p\geq 0$, $q\geq 1$. So we can always find $H$ such that $\omega=\partial H$. In both cases, we have that $u_{\bar{z}}=H_{z\bar{z}}$; therefore, even if $u$ is only locally defined, we can express its derivative $u_{\bar{z}}$ in terms of the Laplacian of a global function. Even if the precise value of such Laplacian depends on the choice of local holomorphic coordinate $z$, its sign does not.

We will now examine three possible choices of $g$ and $\omega$ for which $(M,T^{1,0}M)$ is pseudoconvex, which give rise to three different behaviors. We dub the three resulting classes of CR structures Type I, II, and III, for ease of reference.

\subsubsection{Type I}

If $F=\mathbb{S}^1$ and $g\equiv 1$, then
$$2\lambda_{\theta}(L,\overline{L})=\Im(H_{z\bar{z}})=(\Im(H))_{z\bar{z}}\;.$$
So, $M$ is pseudoconvex as soon as $\Im(H)$ is a subharmonic (or superharmonic) function. Set $W_1:=\{p\in Y:\ \triangle \Im (H)(p)=0\}\subseteq Y$; the Levi-null distribution is supported on
$$S_\levinull=\pi^{-1}(W_1),$$
and thus one may easily verify that \[
TS_\levinull = (d\pi)^{-1}TW_1, 
\]
where $d\pi:TM\rightarrow TY$ is the differential of $\pi$. Denote by $W_2$ the set of points $p\in Y$ where $T_pW_1$ is two-dimensional. By Proposition \ref{upper_semicontinuous_prp}, $W_2$ is closed. The closed sets \[\Sigma_2:=\pi^{-1}(W_2)\subseteq \Sigma_1:=\pi^{-1}(W_1)\subseteq \Sigma_0:=M\] define a weakly regular stratification of $M$, essentially because if $L$ is tangent to the support of $S_\levinull$ at some point, necessarily $TW_1$ is two-dimensional at that point. Thus, by Proposition \ref{weaklydecomp_prp}, the support of the Levi core is contained in $\Sigma_2$. It is not difficult to verify that if $H$ is assumed to be real-analytic, then the extended stratification obtained setting $\Sigma_3:=\emptyset$ is still weakly regular, and therefore the manifold has trivial Levi core. 

\medskip

As a particular case, if $\Im(H)$ is harmonic, then $\lambda_\theta\equiv 0$, i.e. the distribution $H(M)$ is integrable. In this case, $M$ is said to be \emph{Levi-flat} and it is foliated in complex leaves: $H(M)$, the tangent distribution to the leaves, has a natural complex structure.

Such a foliation is described by the $1$-form $\theta$.

\begin{rmk}If $Y$ is compact, then $\Im(H)$ is harmonic if and only if it is constant, hence
$$\theta=dt-\Re(\omega)=dt-\Re(\eta)-\Re(\partial H)$$
and, as $H$ may be assumed to be real-valued, $\Re(\partial H)=dH$; moreover, $d\Re(\eta)=\Re(d\eta)=0$. So $d\theta=0$; the foliation on $M$ is induced by a closed form, so every leaf of the foliation has trivial holonomy; hence, by a result of Sacksteder (cfr [Camacho-Neto, Notes to Chapter V, (2), p. 109]), the foliation has no exceptional minimal set, therefore, if no leaf is compact, then all leaves are dense.

If $Y$ is open, then $\Im(H)$ harmonic implies that locally there exists a real-valued harmonic function $K$ such that $K-i\Im(H)$ is antiholomorphic; therefore $v=H+K-i\Im(H)$ is a real-valued function such that $\partial v=\partial H$. Again $\Re(\partial v)=dv$. So, locally, $\theta=dt-dv$; in conclusion, also in this case, $d\theta=0$.\end{rmk}

\subsubsection{Type II}

If $F=\R$, $g(t)=t$, then
$$2\lambda_\theta(L,\overline{L})=\Im(tH_{z\bar{z}})$$
which is of constant sign if and only if it vanishes, hence $\Im(H)_{z\bar{z}}\equiv0$, i.e. $\Im(H)$ is harmonic and $M$ is again Levi-flat.

\subsubsection{Type III}

If $F=\mathbb{S}^1$, let $g(t)=e^{it}$ and $\omega$ be given by the function $H(z)=ie^{-ih(z)}$ with $h:Y\to\mathbb{R}/2\pi\mathbb{Z}$. Then
$$H_z=h_ze^{-ih}\qquad H_{z\bar{z}}=h_{z\bar{z}}e^{-ih}-i|h_z|^2e^{-ih}\;.$$
So, locally,
$$2\lambda_{\theta}(L,\overline{L})=\Im(e^{i(t-h)}h_{z\bar{z}}-i|h_z|^2e^{i(t-h)})+|f+h_ze^{-ih}|^2\Im(e^{-it}ie^{it})$$
$$=h_{z\bar{z}}\sin(t-h)-|h_z|^2\cos(t-h)+|f|^2+|h_z|^2+2\Re(fh_ze^{i(t-h)})\;.$$
Suppose now also that $Y$ is open, then $f(z)\equiv 0$, so
$$2\lambda_{\theta}(L,\overline{L})=h_{z\bar{z}}\sin(t-h)-|h_z|^2\cos(t-h)+|h_z|^2$$
which is always positive if and only if $h_{z\bar{z}}\equiv 0$, i.e. if and only if $h$ is a harmonic function.

In such case
$$2\lambda_{\theta}(L,\overline{L})=|h_z|^2(1-\cos(t-h))$$
so, $\levinull$ is supported at points where
$$h_z=0\qquad\textrm{or}\qquad t=h\;.$$
The second condition describes a section $s$ of $\pi:M\to Y$, whose image is easily checked to be a complex submanifold of $M$ and its tangent bundle will be contained in $\core(\levinull)$, by Lemma \ref{horizontal_lem}.

On the other hand, as $h$ is harmonic, the points where $h_z=0$ are isolated (if $h$ is nonconstant) and, as $\partial_t$ is never contained in $H(M)$, then the set
$$\{p\in Y\ :\ h_z(p)=0\}\times T\mathbb{S}^1$$
is not contained in $\core(\levinull)$.

In conclusion, $\core(\levinull)$ is the tangent bundle to $s(Y)\subset M$.

\subsection{Real hypersurfaces}\label{hypersurfaces_sec}

Given a complex manifold $X$ of complex dimension $n+1$, any smooth real hypersurface $M\subset X$ has a natural structure of CR manifold of hypersurface type; if $j:M\to X$ is the inclusion, we define
$$T^{1,0}M=j^*(\C TM\cap T^{1,0}X)$$
and $(M, T^{1,0}M)$ is then a CR manifold.

The formal integrability of $T^{1,0}M$ follows from the formal integrability of $T^{1,0}X$ and from the integrability of $\C TM$.

\begin{prp}
A smooth one-form $\theta$ on $M$ is a pseudo-Hermitian structure if and only if there exist a neighborhood $U$ of $M$ in $X$ and a smooth function $r:U\to\R$ such that
$$M=\{p\in U\ : \ r(p)=0\}\qquad dr_p\neq 0\ \forall\, p\in M\qquad \theta=j^*d^cr$$
where $d^c:=i(\partial-\dbar)$. Such a function is called a defining function for $M$ on $U$.
\end{prp}
\begin{proof}It is easy to see that
$$H(M)=j^*(T_pM\cap JT_pM)$$
as the latter is the maximal complex subspace of $T_pM$.

Consider a neighborhood $U$ of $M$ and a smooth defning function $r:U\to\R$.

If $J$ is the complex structure of $X$, we have that $\ker dr_p=T_pM$ for $p\in M$ and $\ker d^cr_p=\ker (dr\circ J)_p=J(\ker dr_p)=JT_pM$; therefore,
$$\ker j^*d^cr_p=j^*(\ker d^cr_p\cap\ker dr_p)=j^*(T_pM\cap JT_pM)=H(M)_p\;.$$
Hence $\theta=j^*d^cr\in\Theta(M,T^{1,0}M)$.

On the other hand, any other $\theta'\in\Theta(M,T^{1,0})$ can be written as $\theta'=g\theta$ with $g\in\mathcal{C}^\infty(M)$ never vanishing; there exist a neighborhood $U'$ of $M$ and a function $f:U'\to \R$ such that $e^f$ extends $g$ from $M$ to $U'$; the function $e^fr$  is again a defining function for $M$ on $U'$, as its zero set in $U'$ is again $M$ and $d(e^fr)_p=r(p)d(e^f)_p+e^{f(p)}dr_p=e^{f(p)}dr_p\neq 0$ for $p\in M$.

We have $d^ce^fr=rd^ce^f+e^fd^cr$, so
$$j^*d^c(e^fr)=j^*(e^fd^cr)=gj^*d^cr=g\theta=\theta'\;.$$
Therefore, every pseudo-Hermitian structure is obtained from a defining function and vice versa.
\end{proof}

We denote by $\mathrm{Def}(M)$ the set of defining functions for $M$; the previous proposition implies that
$$\Theta(M,T^{1,0}M)=\{d^cr\ :\ r\in\mathrm{Def}(M)\}\;.$$

The following is an easy consequence of the last proposition, once we observe that $dd^c=2i\partial\dbar$.

\begin{cor}
The CR manifold $M$ is pseudoconvex if and only if there exists a defining function $r:U\to\R$ such that $\partial\dbar r$ gives a semidefinite Hermitian form on $T^{1,0}M$.
\end{cor}

The equivalence relation $r\sim r'$ if and only if $r=e^fr'$ (up to shrikning the open neighborhood of $M$ on which the functions are defined) partitions $\mathrm{Def}(M)$ in two equivalence classes, by orientability; when $M$ is pseudoconvex, we denote by $\mathrm{Def}_+(M)$ the equivalence class containing a (and hence every) defining function such that $\partial\dbar r$ is positive semidefinite on $T^{1,0}M$. Then
\begin{equation}\label{theta_def_fct}\Theta_+(M,T^{1,0}M)=\{d^cr\ :\ r\in\mathrm{Def}_+(M)\}\;.\end{equation}

Consider, in a complex manifold $X$, a precompact, smoothly bounded domain $\Omega\subseteq X$; its boundary $M=b\Omega$ is a smooth hypersurface in $X$, hence it has a natural CR structure given by $T^{1,0}M$.

We define $\mathrm{Def}(\Omega)$ as the set of defining functions for $\Omega$, i.e. smooth functions $r:U\to\R$ such that
\begin{enumerate}
    \item $U$ is an open set of $X$ with $\Omega\Subset U$
    \item $\Omega=\{x\in U\ :\ r(x)<0\}$
    \item $dr_p\neq 0$ for every $p\in b\Omega$.
\end{enumerate}
The domain $\Omega$ is \emph{pseudoconvex} if $(M, T^{1,0}M)$ is pseudoconvex and $\mathrm{Def}(\Omega)\subseteq \mathrm{Def}_+(M)$, i.e. if for one (and hence for every) defining function $r$ for $\Omega$ the Levi form $\partial\dbar r$ is positive semidefinite on $T^{1,0}M$.

The proof of the following result is essentially contained in \cite{Catlin_global}, even if the statement contained there is a global one. We will make use of Theorem 3 in \cite{Catlin_global}, whose proof is presented in \cite{Catlin_invariants}.

\begin{prp}Let $\Omega\subseteq X$ be a precompact, smoothly bounded, pseudoconvex domain and $M=b\Omega$; if $p\in M$ is of finite type, then there exists a neighborhood $U$ of $p\in M$ such that $U\cap M$ is weakly regular.\end{prp}
\begin{proof}Let $\mathfrak{M}(\cdot)$ be the Catlin multitype, as defined in \cite{Catlin_global}. By \cite[Theorem 3 (2) and (5)]{Catlin_global}, if $p$ is of finite type $\tau(p)$, then there exists a neighborhood $U$ of $p$ such that $\mathfrak{M}(q)\leq \mathfrak{M}(p)$ for all $q\in U\cap M$.
By the Remark following the statement of Theorem 3 in \cite{Catlin_global}, as $\mathfrak{M}(q)$ is bounded by $\tau(p)$, it can only assume a finite number of values; the number of such values depends only on the bound on $\mathfrak{M}(q)$. So, let $M_0<M_1<\ldots < M_{N-1}$ be the possible multitypes for $q\in U$ and set
$$\Sigma_k=\{q\in M\cap U\ :\ \mathfrak{M}(q)\geq M_k\}$$
By the semicontinuity (\!\cite[Theorem 3 (2)]{Catlin_global}) of $\mathfrak{M}$, the sets $\Sigma_k$ are closed in $U\cap M$ and by \cite[Theorem 3(3)]{Catlin_global} each $\Sigma_k\setminus \Sigma_{k+1}$ is contained, up to shrinking $U$, in a submanifold $V_k$ of $U$, with holomorphic dimension $0$, i.e. such that
$$\C T_qV_k\cap\levinull_q=\{0\}\;,\ \ \forall\; q\in V_k\;.$$
Therefore, $M\cap U$ is weakly regular.
\end{proof}

Finally, in the context of hypersurfaces in complex manifolds, Lemma \ref{horizontal_lem} can be reformulated naturally as follows, once we notice that the tangent distribution to a complex analytic set is given pointwise on $M$ by complex subspaces of $T_pX$.

\begin{lem}\label{analytic_lem}Let $M$ be a real smooth hypersurface in a complex manifold $X$, with its natural structure of CR manifold; if $M$ contains a complex analytic set $Z$ (possibly with boundary), then $T^{1,0}Z\subseteq \core(\levinull)$.
\end{lem}

In the last part of this section we show how the CR structures introduced in Section \ref{CRexamples_sec} can be realized as hypersurfaces inside a complex manifold of dimension $2$.

Indeed, consider $X=Y\times\C$, let $z$ be a local variable on $Y$ and $w$ a variable on $\C$ and choose a function $h:Y\to \R/2\pi\Z$; then the function
$$r(z,w)=\Im(we^{-ih(z)})$$
defines a hypersurface $M\subseteq Y\times \C$ which is an embedded realization for type II CR-structures; if we set
$$r(z,w)=|w-e^{ih(z)}|^2-1$$
we obtain the hypersurface of type III. Finally, if
$$r(z,w)=|w|^2-|e^{iH(z)}|^2$$
with $H:Y\to\C$, then the zero set of $r$ is CR-diffeomorphic to a hypersurface of type I.

\medskip

By slightly altering the defining functions, we can add a smooth "cap" to our hypersufaces, making them closed without boundary, and then  obtain relatively compact pseudoconvex domains whose boundary is, on some open set, CR-diffeomorphic to one of the examples. For instance, in the case of type III CR-structures, we set
$$r(z,w)=|w-e^{ih(z)}|^2-(1-\eta(z))$$
for a suitable function $\eta:Y\to[0,+\infty)$. However, particular care should be put into choosing the function $\eta$, if we want to obtain a pseudoconvex domain.

For example, for type III domains, we need the plurisubharmonicity of
$$|w|^2-2\Re(we^{-ih(z)})+\eta(z)\;.$$
As we are working locally, we can multiply everything by $e^{u(z)}$ where $u(z)$ is the harmonic conjugate of $h(z)$:
$$|w|^2e^{u(z)}-\Re(we^{f(z)})+e^{u(z)}\eta(z)\;,$$
where $f(z)=u(z)-ih(z)$ is holomorphic, so that $e^{u(z)}=|e^{f(z)}|$ is clearly plurisubharmonic.

Therefore, we need to make sure that the last term is plurisubharmonic as well; a convenient choice is to set $\eta(z)=\eta_1(h(z))$, with $\eta_1$ convex and increasing. Computations easily show that then $e^{u(z)}\eta(z)$ is plurisubharmonic. We notice that, with such a choice, in order for $\eta$ to be $<1$ on a compact subset of $Y$, it is vital that $Y$ is of dimension $1$.

The well known worm domains introduced by Diederich and Fornaess are instances of this construction, where $Y$ is an annulus and $h(z)=\log|z|^2$.

\medskip

Finally, suppose that $Y$ is an open Riemann surface that can be properly embedded in $\C^2$; it is well known that there exists a tubular neighborhood of $Y\subset\C^2$ which is biholomorphic to a neighborhood of $Y\times\{0\}$ in $Y\times\C$. By rescaling, we can then realize the capped versions of the hypersurfaces of type I and III as boundaries of bounded domains in $\C^2$.

In the type II case, one can construct a bounded domain in $\C^2$ whose boundary contains an open set which is CR-diffeomorphic to a neighborhood (in $M$) of an open set of $Y$. As long as we are careful in adding the caps, the resulting boundary will be again pseudoconvex, if the original hypersurface was.

\subsection{Local maximum sets}\label{local_max_sec}

Let $X$ be a complex manifold, and let $M\subset X$ be a real hypersurface. We proved that the support of the Levi core of $M$ must contain any complex subvariety of $M$. In this section, we show that it also contains every local maximum set contained in $M$. These sets are a generalization of complex analytic sets in terms of the maximum principle for plurisubharmonic functions and, if $M$ is the boundary of a pseudoconvex domain, they coincide with the complement of the weak Jensen boundary. They are also linked with the existence of positive currents of bidimension $(1,1)$ (see for instance \cite{Ohsawa_Sibony_boundedpsh, Sibony_levipb, Bianchi_Mongodi_levicurr}).

\medskip

We will say that a set $K$ is \emph{locally closed} in $X$ if for every $x\in K$ there is a neighborhood $U$ of $x$ in $X$ such that $\overline{U}\cap K$ is closed in $X$.

\begin{dfn}Let $X$ be a complex manifold and let $K \subset X$ be locally closed. We say that $K$ is a \textbf{local maximum set} if
 every $x\in K$ has a relatively compact neighbourhood $U$
such that $\overline{U}$ is contained in a coordinate neighborhood of $X$, $K\cap\overline{U}$ is closed  and for every
function $\psi$ which is plurisubharmonic in a neighbourhood of $\overline{U}$, we have
\[
\max_{\overline{U}\cap K}\psi=  \max_{bU \cap K} \psi
\]
where $bU=\overline{U}\setminus U$.\end{dfn}

For some properties of local maximum sets, see \cite{slodkowski_locmaxprop}. We will use the following result, which is Proposition 2.3, iv) of that paper. 

\begin{prp}\label{slodko_prp}
Let $K$ be a locally closed subset of a complex manifold $X$. Then $K$ is not a local maximum set if and only if there is a $\mathcal{C}^\infty$ strictly plurisubharmonic function $\phi$ on some open set $W$ such that $\phi\vert_{K\cap W}$ has a strict maximum at some point of $K\cap W$.
\end{prp}

Local maximum sets are, in many ways, generalizations of complex analytic sets; their presence influences the behaviour of plurisubharmonic functions around them. The next result generalizes, in this sense, Lemma \ref{analytic_lem}.

\begin{thm} Let $\Omega$ be a precompact, smoothly bounded, pseudoconvex domain in a complex manifold $X$ and consider $M=b\Omega$ as a pseudoconvex CR manifold; if $K\subset M$ is a local maximum set, then $K\subseteq S_{\core(\levinull)}$.\end{thm}

\begin{proof}
Consider the distribution $\mathcal{E}:=\C TK\cap\levinull$. We will show that its support contains $K$. Since this immediately implies that $\mathcal{E}$ is perfect, the conclusion follows from Lemma \ref{ovvio_lem}.

\smallskip

We prove our claim by contradiction, assuming that $x\in K$ and $\C T_xK\cap \levinull_x=\{0\}$, that is, $T_xK\cap JT_xK\cap\levinull^\R_x=\{0\}$. From this, one can write
$$T_xK=H\oplus L\oplus F,$$
where:
\begin{itemize}
    \item $L$ is a complex subspace of $H(M)_x$ such that the restriction of $dd^cr$ to $\C L_x\cap T^{1,0}M$ is positive definite, where $r\in\mathrm{Def}(\Omega)$, 
    \item $H$ is a totally real subspace of $H(M)_x$, i.e. $JH\cap H=\{0\}$, 
    \item $F$ is either $\{0\}$ or a $1$-dimensional complement of $H(M)_x$ in $T_xM$, depending on whether $T_xK\subseteq H(M)_x$ or not.
\end{itemize}

We fix a neighborhood $U$ of $x$ where we have holomorphic coordinates $z_j=x_j+iy_j$ $(j=1,\ldots, n)$ such that (for appropriate integers $k,\ell\geq 0$): 
\begin{itemize}
\item $x$ is the origin, 
\item $H$ is generated by $\left\{\partial_{x_1},\ldots, \partial_{x_k}\right\}, $
\item $\C L_x\cap T_x^{1,0}M$ is generated (over $\C$) by $\left\{\partial_{z_{k+1}},\ldots, \partial_{z_{k+\ell}}\right\}$ 
\item $\partial_{y_n}$ is an inward transversal vector field to $b\Omega$ at $x=0$, and $F$ is contained in the span of $\partial_{x_n}$ (this is possible because $K\subseteq M$ and hence $T_xK\subseteq T_xb\Omega$).
\end{itemize}
From now on, we identify points of $U$ with their coordinates (e.g., $x=0$). It is also useful to group the coordinates as follows:\[
z=(z',z'',z''',z_n),
\]
\[
z'=(z_1,\ldots, z_k)\in \C^k, \quad z''=(z_{k+1},\ldots, z_{k+\ell})\in \C^\ell, 
\]
\[
z'''=(z_{k+\ell+1},\ldots, z_{n-1})\in \C^{n-k-\ell-1}, z_n\in \C.
\]

\medskip

As $T_0K=H\oplus L\oplus F$, by Proposition \ref{tangent_prp}, up to shrinking $U$, we have a manifold $V\subseteq U$ of real dimension $(k+2\ell)$ (if $F=\{0\}$) or $k+2\ell+1$ (if $F$ is $1$-dimensional) such that $T_0V=T_0K$ and $K\subseteq V$; by applying the Proposition directly to the manifold $b\Omega$, we can suppose that $V\subseteq b\Omega$.

The slice
$$U':=\{z\in U\ :\ z'=z'''=0\}$$
is an open set of $\C^{\ell+1}_{z'', z_n}$, and $\Omega':=U'\cap\Omega$ is a pseudoconvex domain in $\C^{\ell+1}$. With respect to the variables $(z'', z_n)$, the domain $\Omega'$ is given by the condition
$$r'(z'',z_n)=r(0,z'',0, z_n)<0,$$
so $dd^cr'$ restricted to the maximal complex subspace of $T_0b\Omega'\subseteq T_0\C^{\ell+1}$ can be identified with $dd^cr$ restricted to $L$. Therefore, $\Omega'$ is strictly pseudoconvex in $0$. Without loss of generality (that is, for an appropriate choice of the coordinate system), we may assume that $\Omega'$ is strictly convex at $0$. Notice that $V\cap U'\subseteq b\Omega\cap U'=b\Omega'$.

\medskip

Shrinking, if needed, $U$ again, we now parametrize $V$ with $T_0V$ (identified with a linear subspace of $\C^n$); we treat the case where $\dim F=1$, the other being analogous. For $z\in U$, recall that $H$ is given by $\Im(z')=z''=z'''=z_n=0$ and $U'$ by $z'=z''=0$; the parametrization of $V$ over $T_0V$ will then be of the form
$$\Im(z')=f'(\Re(z'),z'',x_n),\qquad z'''=f''(\Re(z'),z'',x_n),\qquad y_n=f'''(\Re(z'),z'',x_n)$$
where (for some $C_1>0$) \begin{equation}\label{bigO}
	|f'|,\ |f''|,\ |f'''|\leq C_1(\|\Re(z')\|^2+\|z''\|^2+x_n^2).\end{equation} By the strict convexity of $b\Omega'$,  one easily gets that
$$f'''(\Re(z'),z'',x_n)=q(z'',x_n)+s(\Re(z'))+b(\Re(z');(z'',x_n))+\textrm{h.o.t.},$$
where:\begin{itemize}
	\item $q$ is homogeneous of degree $2$ and positive definite as a real quadratic form on the real vector space $\C^\ell\times \R$,
	\item $s$ is homogeneous of degree $2$,
	\item $b$ is a bilinear map $\R^k\otimes\left(\C^{\ell}\times \R\right)\rightarrow \R$.
	\end{itemize} 
Hence, there is a constant $C_2>0$ such that, for every $\eta>0$,
$$|b(\Re(z');(z'',x_n))|\leq C_2\|\Re(z')\|\|(z'',x_n)\|\leq \frac{C_2}{2}\left(\frac{1}{\eta}\|\Re(z')\|^2+\eta\|(z'',x_n)\|^2\right)$$
therefore, if $C_2\eta/2$ is strictly less than the smallest eigenvalue of $q$, for appropriate constants $a, C_3>0$ (and shrinking once more $U$) we have
\begin{equation}\label{f'''}f'''(\Re(z'),z'',x_n)\geq a\|(z'',x_n)\|-C_3\|\Re(z')\|^2.
\end{equation}

Consider the function
$$\phi(z)=2C_3(\|\Im(z')\|^2-\|\Re(z')\|^2)-y_n+\frac{C_3}{2}\|\Re(z')\|^2+\|\Im(z')\|^2+\frac{a}{2}(\|z''\|^2+x_n^2)+\|z'''\|^2+y_n^2\;.$$
Notice that $\phi$ is strictly plurisubharmonic, because both $\|\Im(z')\|^2-\|\Re(z')\|^2$ and $y_n$ are pluriharmonic. By \eqref{bigO} and \eqref{f'''}, for every $z\in V$ we have (for another constant $C_4>0$, and shrinking $U$ one last time) \begin{eqnarray*}
\phi(z)&\leq& -\frac{a}{2}(\|z''\|^2+x_n^2)-\frac{C_3}{2}\|\Re(z')\|^2+C_4(\|\Re(z')\|^2+\|z''\|^2+x_n^2)^2\\
\\ &\leq& -\frac{a}{4}(\|z''\|^2+x_n^2)-\frac{C_3}{4}\|\Re(z')\|^2, 
\end{eqnarray*} and in particular $\phi(0)=0$ and $\phi(z)<0$ for $V\setminus\{0\}$. Since $K\subseteq V$, by Proposition \ref{slodko_prp} $K$ is not a local maximum set, which is the desired contradiction. The proof is complete.
\end{proof}

\section{D'Angelo forms and D'Angelo class}\label{d'angelo_sec}

In this section we discuss D'Angelo forms and the D'Angelo class. We review the basic theory first (in Section \ref{d'angelo_basic_sec}), and then we introduce the norm-like quantities $\norm_K(\mathcal{D})$, in terms of which our "reduction to the core" theorem is formulated. Sections \ref{de_rham_sec} and \ref{normlike_sec} prepare the ground for the definition of $\norm_K(\mathcal{D})$, discussing a similar notion in the context of complex manifolds and recalling a few basic properties of the Cauchy--Riemann complex on abstract CR manifolds, respectively. 

\subsection{Review of the basic theory}\label{d'angelo_basic_sec}

Let $(M,T^{1,0}M)$ be a connected and orientable CR manifold of hypersurface type, and fix a pseudo-Hermitian structure $\theta\in\Theta(M,T^{1,0}M)$. By orientability, there is a global smooth real vector field $T$ such that $\theta(T)\equiv 1$. A vector field with this property will be said to be \textbf{$\theta$-normalized}. If $T$ is $\theta$-normalized, any other $\theta$-normalized vector field is obtained adding to $T$ an appropriate smooth section of the maximal complex distribution $H(M)$. The choice of a $\theta$-normalized $T$ yields a direct sum decomposition \begin{equation*}
\C TM=T^{1,0}M\oplus T^{0,1}M\oplus \C T.
\end{equation*}
We denote by $\pi_{1,0}$ (resp. $\pi_{0,1}$) the projection onto the first (resp. second) factor. The projection onto the third factor is given by $X\mapsto \theta(X)T$, and $\pi_{1,0}(\overline{X})=\overline{\pi_{0,1}(X)}$.

The next proposition establishes the elementary theory of D'Angelo forms, collecting various known facts.

\begin{prp_dfn}\label{d'angelo_prp} Given a pseudo-Hermitian structure $\theta$ and a $\theta$-normalized vector field $T$, we define, for a smooth complex vector field $X$, \begin{equation}
\alpha_{\theta,T}(X):=\theta([T,X-\theta(X)T]). \label{d'angelo_def}
\end{equation}
The expression \eqref{d'angelo_def} is $C^\infty$-linear in $X$, and $\alpha_{\theta,T}(X)$ is real if $X$ is a real vector field. Hence, it defines a real one-form on $M$. Any such form $\alpha_{\theta, T}$ will be called a \textbf{D'Angelo form} of $(M,T^{1,0}M)$. The following properties hold:\begin{itemize}
    \item[i)] $\alpha_{\theta, T}(T)=0$. Thus, we may write $\alpha:=\alpha_{\theta, T}=\alpha_{1,0}+\alpha_{0,1}$, where $\alpha_{1,0}:=\alpha\circ\pi_{1,0}$, $\alpha_{0,1}:=\alpha\circ\pi_{0,1}$, and $\alpha_{0,1}=\overline{\alpha_{1,0}}$.
    \item[ii)] 	If $X$ is a section of $T^{1,0}M\oplus T^{0,1}M$, then
	\begin{equation*}
	\alpha_{\theta, T}(X)=-\mathcal{L}_T\theta(X),
	\end{equation*}
	where $\mathcal{L}$ is the Lie derivative.
    \item[iii)] The restriction of $\alpha_{\theta, T}$ to $\levinull\oplus \overline{\levinull}$ is independent of the choice of $T$ ($\levinull\oplus \overline{\levinull}$ is the complex distribution whose fiber at $p$ is $\levinull_p\oplus \overline{\levinull_p}$). Hence, we may safely denote it by $\alpha_\theta:\levinull\oplus \overline{\levinull}\rightarrow \C$.
    \item[iv)] If $\theta'=\pm e^f\theta$ is another pseudo-Hermitian structure, where $f\in C^\infty(M,\R)$, then
	\begin{equation*}
	\alpha_{\theta'}=\alpha_{\theta}+ df\quad \text{on } \levinull\oplus\overline\levinull.
	\end{equation*}
	\item[v)] Suppose that $(M,T^{1,0}M)$ is pseudoconvex. If $X_p,Y_p\in \levinull_p\oplus \overline{\levinull_p}$, then $d\alpha_{\theta, T}(X,Y)_{|p}=0$.
	\item[vi)] Let $X$ be a complex manifold, and let $M\subset X$ be a real hypersurface. Let $r\in \mathrm{Def}(M)$, and let $N$ be a $(1,0)$-vector field defined in a neighborhood of $M$ such that $Nr\equiv1$. If $\theta:=d^cr$ and $T:=\frac{1}{2i}(N-\overline{N})$, then
	\begin{equation*}
	\alpha_{\theta, T}(Z)=\overline\partial r([\overline{N}, Z])=2\partial \overline\partial r(Z,\overline{N}) \qquad \forall Z\in T^{1,0}M.
	\end{equation*}
	In view of \eqref{theta_def_fct}, the expressions above exhaust all possible restrictions of D'Angelo forms to the Levi-null distribution.
\end{itemize}
The set of all D'Angelo forms will be called the \textbf{D'Angelo class} of $(M,T^{1,0}M)$ and denoted by $\mathcal{A}_{(M,T^{1,0}M)}$.
\end{prp_dfn}

\begin{proof}\
i) The vanishing of $\alpha_{\theta, T}(T)$ is trivial, and the rest follows from the reality of $\alpha$.

ii) By Leibnitz rule for the Lie derivative plus the identity $\mathcal{L}_TX=[T,X]$, we get
	\begin{equation*}
	\mathcal{L}_T\theta (X)=T(\theta(X))-\theta([T,X]).
	\end{equation*}
If $X$ is a section of $T^{1,0}M\oplus T^{0,1}M$, then $\theta(X)\equiv 0$ and the RHS equals $-\alpha_{\theta, T}(X)$. If $X=T$, then $\theta(X)\equiv 1$ and the RHS vanishes. By i), we are done.

iii) If $T'$ is another $\theta$-normalized vector field and $X_p\in \levinull_p\oplus\overline{\levinull_p}$ is extended to a section of $T^{1,0}M\oplus T^{0,1}M$, then \begin{equation*}
	\alpha_{\theta,T'}(X)=\alpha_{\theta,T}(X)+\theta([T'-T,X]).
	\end{equation*}
Since $T'-T$ is a section of $T^{1,0}M\oplus T^{0,1}M$, the second term vanishes at $p$.
	
iv) Choose a $\theta$-normalized vector field $T$. Then $T':=\pm e^{-f}T$ is $\theta'$-normalized. Extending $X_p\in\levinull_p\oplus\overline{\levinull_p}$ to a section of $T^{1,0}M\oplus T^{0,1}M$ as above, we have $\theta(X)=\theta'(X)\equiv0$ and
\begin{align*}
\alpha_{\theta', T'}(X)&=\theta'([T', X])=\pm e^{f}\theta([\pm e^{-f} T, X])\\
&=\pm e^{f}\theta(\pm e^{-f}[T,X]\mp X(e^{-f})T)\\
&=\alpha_{\theta, T}(X)+Xf.
\end{align*}
In view of iii), this is what we wanted.

v) Extend $X_p, Y_p\in \levinull_p\oplus \overline{\levinull_p}$ to sections of $T^{1,0}M\oplus T^{0,1}M$. Since exterior differentiation and Lie derivative commute, by ii) we get \begin{eqnarray*}
		d\alpha_{\theta, T}(X,Y)&=&-(\mathcal{L}_Td\theta)(X,Y)\\
		&=&-T\left(d\theta(X,Y)\right)+d\theta([T,X],Y)+d\theta(X,[T,Y])\\
		&=&-T\left(d\theta(X,Y)\right)+d\theta([T,X],Y)-d\theta([T,Y],X).
	\end{eqnarray*}
Since $Y_p\in \levinull_p\oplus \overline{\levinull_p}$, $d\theta(Z,Y)_{|p}=0$ for every section $Z$ of $T^{1,0}M\oplus T^{0,1}M$. Hence,
\[d\theta([T,X],Y)_{|p} = \theta([T,X])_{|p} d\theta(T,Y)_{|p} =-\frac{1}{2}\theta([T,X])_{|p} \theta([T,Y])_{|p} ,\]
where we used Cartan's identity. Thus, $d\theta([T,X],Y)_{|p}-d\theta([T,Y],X)_{|p}=0$. \newline
To complete the proof, we need to show that $T\left(d\theta(X,Y)\right)_{|p}=0$. Since $d\theta(X,Y)=d\theta(\pi_{1,0}X,\pi_{0,1}Y)+d\theta(\pi_{0,1}X,\pi_{1,0}Y)$, we may assume for simplicity that $X$ and $Y$ are sections of $T^{1,0}M$ such that $X_p, Y_p\in \levinull_p$, and prove that $T\left(d\theta(X,\overline{Y})\right)_{|p}=0$. In fact, by the standard polarization identity, we may also assume that $X=Y$. By pseudoconvexity, $\frac{1}{2i}d\theta(X,\overline{X})=\lambda_\theta(X,\overline{X})$ is nonnegative, and it vanishes at $p$. The conclusion follows immediately.

vi) Notice that the vector field $T$ is real, tangent to $M$ and satisfies $\theta(T)=1$, where $\theta=d^cr$. Since $[T, Z]$ is also tangent to $M$, $dr([T, Z])=0$ and \begin{eqnarray*}
		\alpha_{\theta,T}(Z)&=&i(\partial r-\dbar r)([T, Z])\\
		&=&-2i\dbar r([T, Z])\\
		&=&\overline\partial r([\overline{N}, Z]),
	\end{eqnarray*} because $[N,Z]$ has type $(1,0)$. Cartan's formula now yields \begin{eqnarray*}
		\overline\partial r([\overline{N}, Z])&=&-2d\overline\partial r(\overline{N}, Z)+\overline{N}(\overline\partial r(Z))-Z(\overline\partial r(\overline{N}))\\
		&=&-2\partial \overline\partial r(\overline{N}, Z)=2\partial \overline\partial r(Z,\overline{N}).
	\end{eqnarray*} \end{proof}

\begin{rmk2}
If $Z$ is a smooth section of $T^{1,0}M$, then by the definition of the Levi form\begin{equation*}
	[Z,\overline{Z}]=\pi_{1,0}[Z,\overline{Z}]-\overline{\pi_{1,0}[Z,\overline{Z}]}-4i\lambda_\theta(Z,\overline{Z})T.
	\end{equation*}
	D'Angelo observes in \cite{DAngelo_iterated} that this gives the identity \begin{align*}
	\lambda_\theta(Z,\overline{[Z,\overline{Z}]})&=\frac{1}{4i}\theta([Z,[Z,\overline{Z}]])\\
	&=\lambda_\theta(Z,\overline{\pi_{1,0}[Z,\overline{Z}]})-Z\left(\lambda_\theta(Z,\overline{Z})\right)+\lambda_\theta(Z,\overline{Z})\theta([T,Z]),
	\end{align*}
	and he is led to give definition \eqref{d'angelo_def}. In fact, this notion already appears in the earlier \cite{DAngelo_finite_type_real_hypersurfaces}, but, as D'Angelo puts it in its 1987 paper, it had been ``little used'' until that moment. One should say that the r\^{o}le played by the ``D'Angelo form'' in \cite{DAngelo_iterated} is also purely instrumental, and it does not appear in the main result of the paper.
	
	It is in the work of Boas and Straube \cite{Boas_Straube_derham} that D'Angelo forms start to play a prominent r\^{o}le, and that the crucial "closure" property v) is observed. Various references for the parts of the Proposition-Definition above are: \begin{enumerate}
	    \item[ii)] \cite[Prop. 9]{DAngelo_iterated}.
	    \item[iii)] \cite[p. 60]{DAngelo_finite_type_real_hypersurfaces} and \cite[Lemma 2.5]{Adachi_Yum}.
	    \item[iv)] \cite[formula (5.84) at p. 158]{Straube_book}.
	    \item[v)] \cite[p. 230]{Boas_Straube_derham}.
	    \item[vi)] \cite[p. 231]{Boas_Straube_derham} and \cite[identity (5.81) at p. 158]{Straube_book}.
	\end{enumerate}
\end{rmk2}

\begin{rmk}
On a Levi-flat CR manifold, the D'Angelo class gives a well-defined leafwise de Rham cohomology class on the Levi foliation; its interpretation in terms of the differential geometric properties of the foliation is discussed, for example, in \cite{Barletta_Dragomir_Duggal_CRfoliation} at the end of Section 3.6.

In presence of a metric on $M$ or at least on the transverse bundle of the foliation, the leafwise differential forms can be identified with "global" differential forms; for example, in \cite{Mongodi_Tomassini_semihol, Mongodi_Tomassini_onecomplete} a condition involving  the D'Angelo form in the Levi-flat case is needed in order to give an appropriate condition of positivity for a metric on the transverse bundle of the foliation.

Following the same lines of thought, one could also spot an instance of D'Angelo forms in the context of Levi foliations in the paper \cite{Forstneric_LaurentT_stein_compacts}. 
\end{rmk}

Assume now that $(M,T^{1,0}M)$ is pseudoconvex. Parts iv) and v) of the Proposition-Definition lead one to think of the D'Angelo class as a sort of cohomology class, that happens to be well-defined only on the distribution $\levinull\oplus \overline\levinull$. In fact, if $\iota:V\hookrightarrow M$ is an embedded submanifold such that $\C T_pV\subseteq \levinull_p\oplus \overline\levinull_p$ for every $p\in M$, then the restricted D'Angelo form $\iota^*\alpha_{\theta, T}$ is a real closed one-form on $V$ independent of $T$, whose associated cohomology class $[\iota^*\alpha_{\theta, T}]\in H^1_{\mathrm{dR}}(V,\R)$ is independent of the choice of pseudo-Hermitian structure (cf. \cite{Boas_Straube_derham}). This happens in particular when $T_pV$ is a complex (i.e., $J_b$-invariant) subspace of $H_p(M)$ for every $p\in M$: in this case $\C T_pV = \C T_pV\cap T^{1,0}M\oplus \C T_pV\cap T^{0,1}M\subseteq \levinull_p\oplus \overline\levinull_p$ (cf. Proposition \ref{horizontal_lem}). A standard application of the Newlander--Nirenberg theorem actually shows that the almost complex structure $(V, J_b)$ is integrable, i.e., it is induced by a structure of complex manifold on $V$.

Our next goal is to attach to any sub-distribution $\mathcal{D}$ of the Levi distribution certain quantities $\norm_K(\mathcal{D})$, defined in terms of the D'Angelo class $\mathcal{A}_{(M,T^{1,0}M)}$ and depending on a parameter $K\in (0,+\infty]$. If $\mathcal{D}$ is the tangent distribution $\C TV$ to a complex submanifold as in the preceding paragraph, $\norm_\infty(\mathcal{D})$ coincides with a certain norm-like function $\norm_V:H^1_{\textrm{dR}}(V,\R)\rightarrow [0,+\infty]$, evaluated at $[\iota^*\alpha_{\theta, T}]$. In the next section, we define $\norm_V$ as a motivating preliminary to the definition of $\norm(\mathcal{D})$, which is given in Section \ref{normlike2_sec}.

\subsection{A norm-like function on the first de Rham cohomology of a complex manifold}\label{de_rham_sec}

Let $X$ be a complex manifold. All the one-forms are assumed to be smooth.

\begin{prp}
If $\alpha$ is a real closed one-form on $X$, then $\dbar \alpha_{1,0}$ induces a Hermitian form on $T^{1,0}X$, that is, \begin{equation*}
\overline{\dbar \alpha_{1,0}(L,\overline{M})}=\dbar \alpha_{1,0}(M,\overline{L})\qquad\forall L,M\in T^{1,0}X.
	\end{equation*}
	\end{prp}

\begin{proof}
Since $\alpha$ is real, $\alpha_{0,1}=\overline{\alpha_{1,0}}$. Since $d\alpha=0$, $\dbar\alpha_{1,0}=-\partial\alpha_{0,1}$. Thus, \begin{equation*}
\overline{\dbar \alpha_{1,0}(L,\overline{M})}=\partial \alpha_{0, 1}(\overline{L},M)=-\dbar\alpha_{1,0}(\overline{L},M)=\dbar\alpha_{1,0}(M,\overline{L}).\end{equation*}\end{proof}

\begin{dfn}\label{quasi_norm_def}
Let $X$ be a complex manifold and $\alpha$ a real closed one-form. Set \begin{equation*}
\norm_{X,0}(\alpha):=\inf\{\lambda \colon \ \lambda> 0\quad \&\quad  \alpha_{1,0}\wedge\alpha_{0,1}\leq\lambda \dbar \alpha_{1,0}\},
\end{equation*}
where the inequality is in the sense of Hermitian forms on $T^{1,0}X$ and $\norm_0(\alpha)=+\infty$ if the set above is empty. If $[\alpha]\in H^1_{\mathrm{dR}}(X,\R)$, then we put \begin{eqnarray*}
\norm_X([\alpha])&:=&\inf\{\norm_{X,0}(\beta)\colon\ [\alpha]=[\beta]\}\\
&=&\inf\{\norm_{X,0}(\alpha+df)\colon\ f\in C^\infty(X,\R)\}.
\end{eqnarray*}
\end{dfn}	

Notice that the two functions $\norm_0$ and $\norm$ stand in the same relation as a norm on a vector space and the induced "quotient norm" on a quotient of that space. The next result shows that $\norm$ satisfies weaker versions of the norm axioms.

\begin{thm}\label{norm_thm}
Let $X$ be a complex manifold and write $\norm=\norm_X$ for simplicity. Then $\norm:H^1_{\textrm{dR}}(X,\R)\rightarrow [0,+\infty]$ satisfies the following properties:
\begin{itemize}
\item[i)] non-degeneracy: $\norm([\alpha])=0$ if and only if $[\alpha]=0$.
\item[ii)] homogeneity: $\norm(t[\alpha])=t\norm([\alpha])$ for every $t\geq 0$ and $[\alpha]\in H^1_{\textrm{dR}}(X,\R)$.
\item[iii)] quasi-triangle inequality: $\norm([\alpha]+[\beta])\leq 2\max\{\norm([\alpha]),\norm([\beta])\}$ for every $[\alpha], [\beta]\in H^1_{\textrm{dR}}(X,\R)$.
\end{itemize}
If $Y$ is another complex manifold and $F:Y\rightarrow X$ is a holomorphic mapping, then the map induced in cohomology $F^*:H^1_{\textrm{dR}}(X,\R)\rightarrow H^1_{\textrm{dR}}(Y,\R)$ is "norm"-decreasing: \[\norm_Y(F^*[\alpha])\leq \norm_X([\alpha])\qquad\forall [\alpha]\in H^1_{\textrm{dR}}(X,\R).\]
\end{thm}

\begin{proof} Homogeneity is easily verified: if $t,\lambda>0$, then $\alpha_{1,0}\wedge\alpha_{0,1}\leq\lambda \dbar \alpha_{1,0}$ if and only if $(t\alpha_{1,0})\wedge(t\alpha_{0,1})\leq\lambda t \dbar(t \alpha_{1,0})$, and therefore $\norm_{X,0}(t\alpha)=t\norm_{X,0}(\alpha)$, from which the homogeneity of $\norm$ follows immediately. The case $t=0$ is trivial (notice that we are using the convention that $0\cdot\infty=0$).

Let us now prove the quasi-triangle inequality. Since $\gamma_{1,0}\wedge\gamma_{0,1}\geq 0$ for every real one-form $\gamma$, considering $\gamma=\alpha-\beta$ we get \[
\alpha_{1,0}\wedge\beta_{0,1}+\beta_{1,0}\wedge\alpha_{0,1}\leq \alpha_{1,0}\wedge\alpha_{0,1}+\beta_{1,0}\wedge\beta_{0,1}.
\]
Assume that $\alpha_{1,0}\wedge\alpha_{0,1}\leq\lambda \dbar \alpha_{1,0}$ and $\beta_{1,0}\wedge\beta_{0,1}\leq\mu \dbar \beta_{1,0}$ for some $\lambda, \mu> 0$, which entails that $\dbar\alpha_{1,0}$ and $\dbar\beta_{1,0}$ are nonnegative. Then\begin{eqnarray*}
&(\alpha_{1,0}+\beta_{1,0})\wedge(\alpha_{0,1}+\beta_{0,1})\leq2\alpha_{1,0}\wedge\alpha_{0,1}+2\beta_{1,0}\wedge\beta_{0,1}\\
\leq& 2\lambda \dbar \alpha_{1,0}+2\mu \dbar \beta_{1,0}
\leq 2\max\{\lambda, \mu\}\dbar (\alpha_{1,0}+\beta_{1,0}).
\end{eqnarray*} Taking the $\inf$ over all allowable values of $\lambda$ and $\mu$, we obtain $\norm_{X,0}(\alpha+\beta)\leq 2\max\{\norm_{X,0}(\alpha), \norm_{X,0}(\beta)\}$. Applying this inequality to $\alpha+df$ and $\beta+dg$ and taking one more $\inf$ over all $f,g\in C^\infty(X,\R)$, we obtain the desired quasi-triangle inequality.

The contraction property is a simple consequence of the facts that, since $F:Y\rightarrow X$ is holomorphic, then pull-back along $F$ preserves the decomposition into $(1,0)$ and $(0,1)$ parts of forms, the wedge operation, the nonnegativity of $(1,1)$-forms, and it commutes with the operator $\dbar$. From these properties one sees that,  if $\alpha_{1,0}\wedge\alpha_{0,1}\leq\lambda \dbar \alpha_{1,0}$ holds for $\lambda>0$ and a real one-form $\alpha$ on $X$, then $(F^*\alpha)_{1,0}\wedge (F^*\alpha)_{0,1}\leq\lambda \dbar (F^*\alpha)_{1,0}$ also holds. The inequality $\norm_{Y,0}(F^*\alpha)\leq \norm_{X,0}(\alpha)$ follows immediately, and this in turn gives the desired inequality.

We are left with the proof of non-degeneracy.
We first establish it on an annulus $X=\{z\in X\colon\ a<|z|<b\}$ ($0<a<b$). The vanishing of $\norm([\alpha])$ is equivalent to the existence of a sequence of smooth real one-forms $\alpha^k$ such that $\alpha^k=\alpha+df^k$ for smooth real-valued $f^k$ and \begin{equation*}
\alpha_{1,0}^k\wedge\alpha_{0,1}^k\leq\frac{1}{k}\dbar \alpha_{1,0}^k.
\end{equation*} Let $\chi$ be a nonnegative test function identically equal to one on a smaller annulus $X':=\{a'<|z|<b'\}$, where $a<a'<b'<b$. Since $d\alpha^k=0$, we have the identities $\dbar\alpha_{1,0}^k=d\alpha_{1,0}^k=-d\alpha_{0,1}^k$. Observing that, if $\gamma$ is a real one-form, then $\int_X i\gamma_{1,0}\wedge\gamma_{0,1}\geq 0$, an integration by parts yields
\begin{eqnarray*}
\int_X i\chi^2 \alpha_{1,0}^k\wedge\alpha_{0,1}^k&\leq&\frac{1}{k}\int_X i\chi^2 \dbar\alpha_{1,0}^k\\
&\leq&\frac{1}{2k} \int_Xi\chi^2 d( \alpha_{1,0}^k-\alpha_{0,1}^k)\\
&=&\frac{1}{k}\int_X i\left(\chi\alpha_{1,0}^k\wedge d\chi+d\chi\wedge \chi \alpha_{0,1}^k\right)\\
&\leq&\frac{1}{k}\int_X i\left(\chi^2\alpha_{1,0}^k\wedge \alpha_{0,1}^k+\partial\chi\wedge\dbar\chi\right).
\end{eqnarray*}
Hence, \begin{equation*}
\int_{X'}i\alpha_{1,0}^k\wedge\alpha_{0,1}^k\leq\int_X\chi^2 i\alpha_{1,0}^k\wedge\alpha_{0,1}^k\leq \frac{1}{k-1}\int_Xi\partial\chi\wedge\dbar\chi.
\end{equation*}
Thus, $\alpha^k=\alpha+df^k$ vanishes in $L^2(X')$ as $k$ tends to $\infty$. In particular, the $df^k$ are bounded in $L^2(X')$ uniformly in $k$ and by Poincar\'e inequality, modulo adding an appropriate real constant to $f^k$, we may assume that the $f^k$ are uniformly bounded in $L^2(X')$. Passing to a subsequence, we may assume that $f^k$ converges weakly, hence in the sense of distributions, to $f\in L^2(X')$. This in turn implies that $\alpha+df^k$ converges in the sense of currents (on $X'$) to $\alpha+df$. Since we already proved that $\alpha+df^k$ vanishes in $L^2(X')$, we necessarily have $\alpha=-df$ on $X'$. Since $\alpha$ is smooth, $f$ must be smooth on $X'$. Thus $\alpha$ is exact on $X'$, which in turn implies exactness of $\alpha$ on the whole of $X$.

Let now $X$ be arbitrary and $\norm([\alpha])=0$. We argue by contradiction, assuming that $[\alpha]\neq0$. By the de Rham isomorphism theorem, there exists a smooth closed path $\gamma_0:\mathbb{S}^1\rightarrow X$ such that $\int_{\gamma_0}\alpha\neq 0$. We now use the fact that any smooth mapping between real analytic manifolds can be uniformly approximated by a real analytic mapping (see \cite{shiga}, Section 1). In particular, there is a real analytic path $\gamma_1$ homotopic to $\gamma_0$. Thus, $\int_{\gamma_1}\alpha\neq 0$. Identifying $\mathbb{S}^1$ with the unit circle in the complex plane, we can analytically continue it to a neighborhood of $\mathbb{S}^1$, obtaining a holomorphic mapping $F:A:=\{1-\delta<|z|<1+\delta\}\rightarrow X$ such that $F_{|\mathbb{S}^1}=\gamma_1$. This immediately implies that $\int_{\mathbb{S}^1}F^*\alpha\neq 0$, i.e., $[F^*\alpha]\in H^1_{\mathrm{dR}}(A)$ is nontrivial. By the first part of our proof $\norm_A([F^*\alpha])>0$, and by the contraction property of $\norm$ along holomorphic mappings, we conclude that $\norm_X([\alpha])>0$, the sought-after contradiction. \end{proof}

On complex manifolds with vanishing $H^{0,1}_{\dbar}X$, there is an alternative description of $\norm$ in terms of pluriharmonic functions. This is based on known facts, which we rapidly recall. Denote by $\PH(X)$ the set of smooth, real-valued and pluriharmonic functions on $X$, that is, \begin{equation*}
\PH(X):=\{h\in C^\infty(X,\R)\colon dd^ch=0 \},
\end{equation*} where as usual $d^c=i(\dbar-\partial)$. We have the linear mapping \begin{eqnarray*}
\PH(X) &\longrightarrow & H^1_{\mathrm{dR}}(X)\\
h &\longmapsto & [d^c h]
\end{eqnarray*}
If $h=\Re(H)$, with $H$ holomorphic, then \begin{equation*}
d^ch=i(\dbar-\partial)\left(\frac{H+\overline{H}}{2}\right)=\frac{1}{2i}(dH-d\overline{H})=d(\Im(H)),
\end{equation*}
and therefore $[d^ch]=0$ and the mapping above descends to the quotient \begin{equation}\label{PH_deRham}
\PH(X)/\Re(\OO(X)) \rightarrow  H^1_{\mathrm{dR}}(X).
\end{equation}
This mapping is always injective. In fact, $d^ch=df$, with $f$ and $h$ real-valued, happens if and only if $h+if$ is holomorphic, that is, only if $h\in\Re(\OO(X))$.

\begin{lem}\label{hodge_theory_lem}
Assume that $H^{0,1}_{\dbar}X=0$. Then the mapping \eqref{PH_deRham} is an isomorphism.
\end{lem}

\begin{proof}
Let $\alpha$ be a real closed one-form. Since $\dbar \alpha_{0,1}=0$ and $H^{0,1}_{\dbar}X=0$, there exists a smooth complex valued function $g$ such that $\dbar g=\alpha_{0,1}$. Since $\alpha$ is real, this implies $\alpha_{1,0}=\partial \overline{g}$ and we may write \begin{equation*}
\alpha=\partial \overline{g}+\dbar g = d^c(\Im(g))+d\Re(g).
\end{equation*}
Since $\partial \alpha_{1,0}=0$, $\partial \dbar g=0$ and $\partial \dbar\overline{g}=0$. In particular, $\Im(g)\in \PH(X)$ and the mapping \eqref{PH_deRham} is surjective.
\end{proof}

In particular, any class $[\alpha]$ in $H^1_{\mathrm{dR}}(X)$ is represented by $d^ch$ for some $h\in \PH(X)$ and \begin{eqnarray}
\norm([\alpha])&=&\inf\{\lambda\geq 0\colon \exists f\in C^\infty(X,\R)\text{ s.t. } \partial (f+ih)\wedge\dbar(f-ih)\leq\lambda \dbar \partial f\}\notag\\
&=&\inf\{\lambda\geq 0\colon \exists f\in C^\infty(X,\R)\text{ s.t. } \partial (h+if)\wedge\dbar(h-if)\leq\lambda \partial \dbar f\}\label{norm_pluriharm}
\end{eqnarray}

\subsection{Remarks on the Cauchy--Riemann complex on abstract CR manifolds}\label{normlike_sec}

To define the quantities $\norm_K(\distr)$ on an abstract CR manifold of hypersurface type, we need to consider the Cauchy-Riemann complex on an abstract CR manifold; to keep the exposition self-contained and to avoid a too long detour, we will only recall a few notions and refer the reader to \cite[Chapter 8]{Boggess-CRmanifolds}, or \cite[Section 1.7]{dragomir_tomassini} for a metric-free approach.

We consider a CR manifold $(M,T^{0,1}M)$ of hypersurface type and endow $\C TM$ with a Hermitian metric $h$ such that $T^{0,1}M\perp T^{1,0}M$; we denote by $E$ the orthogonal complement of $T^{0,1}M\oplus T^{1,0}M$. The space of $(p,q)$-forms on $M$ is the space of sections of the bundle
$$\Lambda^{p,q}T^*M=\Lambda^p(T^{1,0}M\oplus E)^*\widehat{{\otimes}}\,\Lambda^q(T^{0,1}M)^*\;.$$
In particular, $(1,0)$-forms are sections of $(T^{1,0}M\oplus E)^*$ and $(0,1)$-forms are sections of $(T^{0,1}M)^*$.

The metric $h$ gives an orthogonal decomposition
$$\Lambda^rT^*M=\Lambda^{0,r}T^*M\oplus\cdots\oplus\Lambda^{r,0}T^*M\;.$$

We denote by $\mathcal{E}^{p,q}_M(U)$ the space of sections of $\Lambda^{p,q}T^*M$ on $U$; we will omit the open set $U$, when it is unimportant.

\begin{rmk}
In general, given a $1$-form $\alpha$ on $M$, we have that $$\alpha=\alpha_{1,0}+\alpha_{0,1}$$
however, if $\alpha$ is real, we cannot say that $\alpha_{1,0}=\overline{\alpha_{0,1}}$. This is only true when we evaluate the two forms on a $(1,0)$ vector-field, i.e. given $X\in \Gamma(M, T^{1,0}M)$,
$$\alpha_{1,0}(X)=\overline{\alpha_{0,1}}(X)\;.$$
It is true that $\overline{\alpha_{0,1}}\in\mathcal{E}^{0,1}_M$.
\end{rmk}

The operator
$$\dbar:\mathcal{E}^{p,q}_M\to\mathcal{E}^{p,q+1}_M$$
is defined as $\dbar=\pi_{p,q+1}\circ d$, where $\pi_{p,q+1}$ is the orthogonal projection from $\Lambda^{p+q+1}T^*M$ to $\Lambda^{p,q+1}T^*M$.

\begin{rmk}
Given $\alpha$ a $1$-form on $M$ and $X,Y$ sections of $T^{1,0}M$, then
$$\dbar\alpha_{1,0}(X,\overline{Y})=d\alpha_{1,0}(X,\overline{Y})\;.$$
Therefore, by Cartan's identity
$$2\dbar\alpha_{1,0}(X,\overline{Y})=-\overline{Y}(\alpha_{1,0}(X))-\alpha_{1,0}([X,\overline{Y}])\;.$$
Due to the definition of $(1,0)$-forms, in general $\alpha_{1,0}([X,\overline{Y}])\neq \alpha([X,\overline{Y}]_{1,0})$; however, if $X_p,Y_p\in\levinull_p$, then $[X,\overline{Y}]_p\in T^{1,0}M\oplus T^{0,1}M$, so
$$\alpha_{1,0}([X,\overline{Y}])_p= \alpha([X,\overline{Y}]_{1,0})_p\;.$$
\end{rmk}

In view of \cite[Section 8.3, Theorem I]{Boggess-CRmanifolds}, given an embedding $j:M\to X$ of $M$ as a CR hypersurface in a complex manifold $X$, the bundles $\Lambda^{p,q}T^*M$ are isomorphic to certain subbundles of $\Lambda^{p+q}T^*X$ via $j^*$ (described in \cite[Section 8.1]{Boggess-CRmanifolds}), which we will denote by $\Lambda^{p,q}T^*j(M)$.

We note that, given a section $\alpha$ of $\Lambda^{1,0}T^*j(M)\oplus\Lambda^{0,1}T^*j(M)$ and given any extension $\tilde{\alpha}$ to a neighbourhood of $j(M)$ in $X$, $\tilde{\alpha}_{0,1}\vert_{T^{0,1}j(M)}=\alpha_{0,1}$, where in the left hand side we took the $(0,1)$ component with respect to the Dolbeault splitting of $T^*X$.

\medskip

We define a $\dbar_{j(M)}$ on the embedded hypersuface via the relation $j^*\dbar_{J(M)}=\dbar j^*$. In particular, given $\alpha$ a section of $(T^{1,0}X)^*\vert_{j(M)}$ such that $j^*\alpha\in\Lambda^{1,0}T^*M$ and $X,Y$ sections of $T^{1,0}M$, then
$$\dbar(j^*\alpha)(X,\overline{Y})=\dbar r\wedge \dbar\tilde{\alpha}(j_*X, j_*\overline{Y}, \overline{N})$$
where $r$ is a defining function for $j(M)$ in $X$, $N$ is a $(1,0)$ vector-field such that $Nr\equiv 1$ and $\tilde{\alpha}$ is an extension of $\alpha$ to a neighborhood of $j(M)$.

We recall the following elementary variation on Cartan's identity (see also \cite[Lemma 4.4]{Yum_invariance}).

\begin{lem}\label{complex_cartan_lem}
	Let $\alpha$ be a real one-form, and let $Z$ and $W$ be vector fields of type $(1,0)$. We have the identities
	\begin{eqnarray*}
	2\partial \alpha(Z,\overline{W})&=&Z\left(\alpha(\overline{W})\right)-\alpha([Z, \overline{W}]_{0,1}),\\
	2\dbar \alpha(Z,\overline{W})&=&-\overline{W}\left(\alpha(Z)\right)-\alpha([Z, \overline{W}]_{1,0}).
	\end{eqnarray*}
\end{lem}

By the definition of wedge product and Cartan's identity,
$$2\dbar r\wedge \dbar\tilde{\alpha}(j_*X, j_*\overline{Y}, \overline{N})=2\dbar\tilde{\alpha}(j_*X, j_*\overline{Y})=-(j_*\overline{Y})(\tilde{\alpha}(j_* X)) - \tilde{\alpha}([j_*X,j_*\overline{Y}]_{0,1})$$
where $[j_*X,j_*\overline{Y}]_{0,1}$ denotes the projection on $T^{0,1}X$.

If $(j_*X)_x, (j_*Y)_x\in\levinull_x$ (i.e. if $X_p, Y_p\in\levinull_p$ with $x=j(p)$), then the projection on $T^{0,1}X$ is the same as the projection on $T^{0,1}j(M)$.

In conclusion, if $X,Y$ are sections of $T^{1,0}M$ such that $X_p, Y_p$ belong to $\levinull_p$, we have a CR embedding $j:M\to X$ and $\alpha$ is a section of $(T^{1,0}X)^*\vert_{j(M)}$ such that $j^*\alpha\in\Lambda^{1,0}T^*M$, then the three expressions
$$\begin{array}{cl}\dbar j^*\alpha(X,\overline{Y})_p&\textrm{(intrinsic $\dbar$ on $M$)}\\
(\dbar r\wedge\dbar \tilde{\alpha})(j_*X, j_*\overline{Y}, \overline{N})_p&\textrm{(extrinsic $\dbar$ on $j(M)$)}\\
\dbar \tilde{\alpha}(j_*X, j_*\overline{Y})_p&\textrm{(restriction of the $\dbar$ from $X$)}\end{array}$$
all coincide, where $\tilde{\alpha}$, $N$, $r$ are as above.

In what follows, given $\alpha$ a $1$-form on $(M, T^{1,0}M)$ we will denote by $\dbar\alpha$ the $(1,1)$-form $\dbar\alpha_{1,0}$, where $\alpha_{1,0}$ is the $(1,0)$-component of $\alpha$; we will be mainly (almost exclusively) interested in the Hermitian form induced by $\dbar\alpha$ on $\levinull$, where all the possible interpretations of the symbol $\dbar$ coincide.

Likewise, given $f\in\mathcal{C}^\infty(M)$, we will denote by $\partial\dbar f$ the $(1,1)$-form $-\dbar (df)_{1,0}$, which again defines a Hermitian form on $\levinull$.

As pointed out before, in the embedded case, we can compute the Hermitian form $\dbar\alpha$ on $\levinull$ by extending $\alpha$ to the ambient complex manifold and then taking the $\dbar$ in the classical sense.

\subsection{Definition of $\norm_K(\mathcal{D})$}\label{normlike2_sec}

We now possess all the necessary tools to define the quantities $\norm_K(\mathcal{D})$ in full generality; we start by studying the Hermitian forms $\partial\alpha$, where $\alpha$ is a D'Angelo form. Let $(M,T^{1,0}M)$ be a pseudoconvex CR manifold of hypersurface type.

\begin{lem}\label{dbar_alpha_lem}[cf. \cite[Lemma 2.6 and Lemma 2.9]{Adachi_Yum}]
	Let $\alpha:=\alpha_{\theta, T}\in \mathcal{A}_{M,T^{1,0}M}$ be a D'Angelo form. Then the Hermitian form $\dbar\alpha$ on $\levinull$ is independent of the choice of $T$. If $\theta, \theta'=e^f\theta\in \Theta_+(M)$, where $f\in C^\infty(M)$, and $\alpha$, $\alpha'$ are the associated D'Angelo forms, then we have \begin{equation}\label{hermitian_form_dependence}
	\dbar\alpha'(Z,\overline{Z})_{|p}= \dbar\alpha(Z,\overline{Z})_{|p}-\partial \dbar f(Z,\overline{Z})_{|p} \qquad\forall Z\textrm{ s.t. }Z_p\in \levinull_p.
	\end{equation}
\end{lem}

\begin{proof}
	Extend $Z_p\in \levinull_p$ to a smooth section of $T^{1,0}M$, as usual. Since $\alpha(Z)=\theta([T,Z])$, independence of the choice of $T$ boils down to checking that
	\begin{equation*}
	\overline{W}\left(\theta([X,Z])\right)_{|p}+\theta([X,[Z,\overline{W}]])_{|p}=0
	\end{equation*} for every $X$ such that $\theta(X)\equiv0$. By Cartan's formula and Jacobi identity, \begin{eqnarray*}
		&&\overline{W}\left(\theta([X,Z])\right)+\theta([X,[Z,\overline{W}]])\\
		&=&2d\theta(\overline{W},[X,Z])+[X,Z]\left(\theta(\overline{W})\right)+\theta([\overline{W},[X,Z]])+\theta([X,[Z,\overline{W}]])\\
		&=&2d\theta(\overline{W},[X,Z])-\theta([Z,[\overline{W},X]]).
	\end{eqnarray*}
	Since $Z_p\in \levinull_p$, $[X,Z]_p\in \C H_p(M)$, and therefore $d\theta(\overline{W},[X,Z])_{|p}=0$, because $W_p\in \levinull_p$. Analogously, $[\overline{W},X]_p\in \C H_p(M)$ and therefore $\theta([Z,[\overline{W},X]])_{|p}=0$ too.
	
	We are left with the proof of \eqref{hermitian_form_dependence}. By part iv) of Definition-Proposition \ref{d'angelo_prp}, \begin{equation*}
	\alpha'=\alpha+df
	\end{equation*} on sections of $\C H(M)$, from which we immediately get \eqref{hermitian_form_dependence}.
\end{proof}

To proceed, we need an analogue of the concept of comass for a $1$-form acting on a distribution $\mathcal{D}$.

\begin{dfn}[$(h,\mathcal{D})$-size]\label{hDsize_dfn} Let $M$ be a real smooth manifold. Fix a Hermitian metric $h$ on $\C TM$ and a complex distribution $\mathcal{D}\subseteq\C TM$. If $\alpha$ is a one-form on $M$, we define the $(h,\mathcal{D})$-size of $\alpha$ as \begin{equation*}
||\alpha||_{h,\mathcal{D}}:=\sup\{|\alpha(Z_p)|\colon\ p\in M, \ Z_p\in\mathcal{D}_p \text{ s.t. } |Z_p|=1\},
\end{equation*}	
where $|Z_p|^2=h_p(Z_p,Z_p)$.
\end{dfn}

\begin{rmk} If $h$ and $h'$ are two comparable Hermitian metrics on $\C TM$, i.e., if there is a positive constant $C$ such that
$$C^{-1}h_p(Z_p,Z_p)\leq h'_p(Z_p,Z_p)\leq Ch_p(Z_p,Z_p)\qquad \forall p\in M,\quad \forall _p\in \C T_pM, $$
then the two corresponding notions of size given by Definition \ref{hDsize_dfn} are comparable:
$$C^{-\frac{1}{2}}\|\alpha\|_{h,\levinull}\leq \|\alpha\|_{h',\levinull}\leq C^{\frac{1}{2}}\|\alpha\|_{h,\levinull}\;.$$
Notice that, when $M$ is compact, all the Hermitian metrics on $\C TM$ are comparable.
\end{rmk}

The next key definition provides a way to measure the size of the D'Angelo class on an arbitrary sub-distribution of the Levi null distribution.

\begin{dfn}
	\label{adachi_yum_dfn}
	Let $(M,T^{1,0}M)$ be a pseudoconvex CR manifold of hypersurface type. Fix a Hermitian metric $h$ on $\C TM$.
	
	If $\alpha\in \mathcal{A}_M$ is a D'Angelo form and $\distr$ is a sub-distribution of the Levi-null distribution $\levinull$, we put
	\begin{equation*}
	\norm(\alpha;\distr):=\inf\left\{t>0\colon\  \left(\overline{\alpha_{0,1}}\wedge \alpha_{0,1}\right)<t\dbar \alpha \text{  on } \distr\right\}.
	\end{equation*}
If $K\in(0,+\infty]$, then we define \begin{equation*}
	\norm_K(\distr):=\inf\left\{\norm(\alpha;\distr)\colon\  \alpha\in \mathcal{A}_M \quad \&\quad  ||\alpha||_{h,\levinull}< K\right\}.
	\end{equation*}
We write $\norm(\distr):=\norm_{+\infty}(\distr)=\inf\left\{\norm(\alpha;\distr)\colon\ \alpha\in \mathcal{A}_M\right\}$.
\end{dfn}

Notice that, if $M$ is compact, then \begin{equation}\label{norm}
\norm(\distr)=\inf_K\norm_K(\distr),
\end{equation} for any sub-distribution $\distr\subseteq\levinull$. This is a simple consequence of the fact that, when $M$ is compact, any D'Angelo form has finite $(h,\levinull)$-size.

To appreciate the meaning of Definition \ref{adachi_yum_dfn}, the reader may observe that the vanishing of $\norm(\distr)$ is equivalent to the existence of a sequence of D'Angelo forms $\alpha^{(m)}$ with the property that the inequality of Hermitian forms $\left(\overline{\alpha^{(m)}_{0,1}}\wedge \alpha^{(m)}_{0,1}\right)<m^{-1}\dbar \alpha^{(m)}$ holds on $\distr$, while the vanishing of $\norm_K(\distr)$ for some $K<+\infty$ amounts to the existence of a sequence $\alpha^{(m)}$ as above, with the additional property that the $(h,\levinull)$-size $||\alpha^{(m)}||_{h,\levinull}$ remains bounded as $m$ goes to $+\infty$. In view of the Remark above, the existence of a $K<+\infty$ such that $\norm_K(\distr)=0$ is independent of the choice of the metric $h$.

\begin{prp}\label{norm_complex_submfd}
Let $(M,T^{1,0}M)$ be a pseudoconvex CR manifold of hyperstype and $\alpha\in \mathcal{A}_{(M,T^{1,0}M)}$ be a D'Angelo form on $M$.  Suppose that $M$ contains a complex manifold $V$ (i.e. containing a differentiable manifold $V$ whose tangent space is a complex subbundle of $H(M)$) such that $[\alpha\vert_V]\in H^1_{\mathrm{dR}}(V)$ is non trivial.

Then $\norm(\core(\levinull))>0$ (hence obviously $\norm_K(\core(\levinull))>0$ for all $K>0$).
\end{prp}
\begin{proof}As a first remark, one should notice that $\C TV\cap T^{1,0}M\subseteq\levinull$ and, in fact, as $\C TV\cap T^{1,0}M$ is perfect, $\C TV\cap T^{1,0}M\subseteq\core(\levinull)$, by Lemma \ref{horizontal_lem}.

By Theorem \ref{norm_thm}-i), as $[\alpha\vert_V]$ is non trivial in $H^1_{\mathrm{dR}}(V)$, $\norm_V([\alpha\vert_V])>0$. Unravelling the definitions, this is equivalent to
$$\inf\{t>0\ :\ (\overline{\alpha_{0,1}}\wedge \alpha_{0,1})<t\dbar\alpha\ \mathrm{ on }\ \C TV\ \mathrm{ with }\ \alpha\in\mathcal{A}_M\}>0\;.$$
As it is obvious, if we take the infimum on a larger distribution (namely $\core(\levinull)$), it will at most increase, thus showing that $\norm(\core(\levinull))>0$.
\end{proof}

In general, if $\mathcal{E}\subseteq\mathcal{D}$ are distributions, then $\norm_K(\mathcal{E})\leq\norm_K(\mathcal{D})$ for all $K\in (0,+\infty]$.

\begin{rmk}If $M$ is a type III CR structure, as described in Section \ref{CRexamples_sec}, with local coordinates $t, z$, a D'Angelo form is given by 
$$\alpha=\Re(g'\omega)=\Re(ie^{it}u(z)dz)=\Im(e^{it}u(z)dz)\;.$$
As observed in Section \ref{CRexamples_sec}, type III CR structures contain a complex submanifold given by a section $s:Y\to M$ of the projection $M\to Y$; if we write $u(z)=f(z)+h_ze^{-ih}$, then the image of such section is described by the equation $t=h(z)$ and, on it, the D'Angelo form given before becomes
$$\alpha\vert_{s(Y)}=\Im(f(z)e^{ih}dz)+d^ch\;.$$
In the embedded realization given at the end of Section \ref{hypersurfaces_sec}, we consider only an open set of the compact Riemann surface $Y$, where we can suppose that $f(z)\equiv 0$, so that $\alpha=d^ch$; we see that, according to Proposition \ref{norm_complex_submfd}, $\norm(\core(\levinull))$ will be positive as soon as $d^ch$ does not give a trivial cohomology class, which is the case, for example, in the Worm domain.
\end{rmk}

\begin{prp}\label{norm_technical_prp}
Assume that $(M, T^{1,0}M)$ is a compact pseudoconvex real hypersurface in a complex manifold $X$, and suppose that there exists a smooth strictly plurisubharmonic function in a neighborhood of $M$ in $X$. Then the quantity $\norm(\distr)$ is finite for every distribution $\distr\subseteq \levinull$, and there exists $K<+\infty$ such that $\norm_K(\distr)$ is also finite. Moreover, if
\begin{equation*}
\widetilde\norm(\alpha;\distr):=\inf\left\{t>0\colon\ \left(\overline{\alpha_{0,1}}\wedge \alpha_{0,1}\right)\leq t\dbar \alpha \text{  on } \distr\right\},
\end{equation*} then we have 	\begin{equation*}
\norm(\distr)=\inf\left\{\widetilde\norm(\alpha;\distr)\colon\ \alpha\in \mathcal{A}_M\right\}
\end{equation*}
and
\begin{equation*}
\norm_K(\distr)=\inf\left\{\widetilde\norm(\alpha;\distr)\colon \alpha\in \mathcal{A}_M \quad \&\quad  ||\alpha||_{h,\levinull}< K\right\}.
\end{equation*}
\end{prp}

\begin{proof}
Let $\alpha$ be any D'Angelo form and $f$ a smooth strictly plurisubharmonic function defined in a neighborhood of $M$. Then $\beta=\alpha-adf$ is a D'Angelo form for every $a>0$ (by Proposition \ref{d'angelo_prp}, part iv)) and the associated Hermitian form is $\dbar\beta=\dbar \alpha+a\partial\dbar f$, by Proposition \ref{dbar_alpha_lem}. If $a$ is large enough, this form is positive, and the inequality of quadratic forms \begin{equation*}
\left(\overline{\beta_{0,1}}\wedge \beta_{0,1}\right)<t\dbar\beta
\end{equation*}
holds on $\levinull$ for $t$ large enough, showing that $\norm(\beta;\levinull)$ is finite. The statement is an immediate consequence of this fact.
\end{proof}

The hypotheses of the last proposition hold, for instance, when $M$ is the boundary of a precompact, smoothly bounded, pseudoconvex domain in a Stein manifold.

\section{Reduction to the core}\label{reduction_sec}

We can finally state and prove the main result of the paper. 

\begin{thm}\label{reduction_thm} Let $(M,T^{1,0}M)$ be a compact pseudoconvex CR manifold of hypersurface type. For every $K\in(0,+\infty]$, we have
	\begin{equation*}
	\norm_K(\levinull)=\norm_K(\core(\levinull)).
	\end{equation*}
\end{thm}

\begin{proof}
By \eqref{norm}, it is enough to consider the case of $K$ finite. We argue by transfinite induction. We need to show two facts: \begin{itemize}
		\item[(A)] If $\distr$ is a sub-distribution of the Levi-null distribution, then $\norm_K(\distr)=\norm_K(\distr')$.
		\item[(B)] If $\lambda_1$ is an ordinal and $\{\distr_\lambda\}_{\lambda<\lambda_1}$ is a decreasing sequence of sub-distributions of the Levi-null distribution, then $\inf_{\lambda<\lambda_1}\norm_K(\distr_\lambda)=\norm_K(\bigcap_\lambda\distr_\lambda)$.
	\end{itemize}

Let us prove (A). Notice first that $\norm_K(\distr)\geq\norm_K(\distr')$ is trivial. To prove the reverse inequality, fix $t>\norm_K(\distr')$. Then there exists a D'Angelo form $\alpha$ such that $\norm(\alpha;\distr')<t$ and $||\alpha||_\levinull<K$. If we show that there exists another one $\gamma$ such that $\norm(\gamma;\distr)<t$ and $||\gamma||_\levinull<K$, then the thesis will follow.

Let $\beta$ be an arbitrary D'Angelo form. Consider the set \begin{equation*}
V(\beta):=\left\{L\in \C TM\colon \left(\overline{\beta_{0,1}}\wedge \beta_{0,1}\right)(L,\overline{L})<t\dbar \beta(L,\overline{L})\right\}.
\end{equation*}

The set $V(\beta)$ is clearly open and conical in $T^{1,0}M$.

Recall that, by part iv) of Proposition \ref{d'angelo_prp}, $\alpha+df\in \mathcal{A}_M$ for every $f\in C^\infty(M,\R)$. Denote by $J$ the ideal of smooth real-valued functions vanishing on the support of $\distr$. We claim that \begin{equation}\label{claimed_cover}
\bigcup_{g\in J} V(\alpha-d(g^2))\supseteq \distr \setminus 0.
\end{equation}
Let us prove \eqref{claimed_cover}. By assumption, $\norm(\alpha;\distr')<t$, that is, $\distr'\setminus 0\subseteq V(\alpha)$, so our task reduces to proving that for every $p\in S_\distr$ and $L_p\in \distr_p\setminus \C T_pS_\distr$, $L_p\in V(\alpha-d(g^2))$ for an appropriate choice of $g\in J$.

If $L_p\notin \C T_pS_\distr$, there exists $g\in J$ such that $L_pg\neq 0$. Let $\beta=\alpha-d((cg)^2)$, where $c\in \R$. Notice that $d(g^2)=0$ at points of $S_\distr$. Lemma \ref{dbar_alpha_lem} gives \begin{eqnarray}
&&t\dbar \beta(L_p,\overline{L_p})-\left(\overline{\beta_{0,1}}\wedge \beta_{0,1}\right)(L_p,\overline{L_p})\notag\\
&=&tc^2\overline{L}L(g^2)(p)+t\dbar \alpha(L_p,\overline{L_p})-\left(\overline{\alpha_{0,1}}\wedge\alpha_{0,1}\right)(L_p,\overline{L_p})\notag\\
&=&2tc^2|L_pg|^2+t\dbar \alpha(L_p,\overline{L_p})-\left(\overline{\alpha_{0,1}}\wedge \alpha_{0,1}\right)(L_p,\overline{L_p}).\label{quadratic_h}\end{eqnarray}
Choosing $c$ large, we may guarantee that \eqref{quadratic_h} is positive, that is, $L_p\in V(\alpha-d((cg)^2))$, as we wanted. The proof of \eqref{claimed_cover} is complete.

Now, by Proposition \ref{compactness_prp} there exists $g_1,\dots, g_N$ in the ideal $J$ such that \begin{equation}\label{claimed_cover_2}
\bigcup_{j=1}^N V(\alpha-d(g_j^2))\supseteq \distr \setminus 0.
\end{equation}
Define $\widetilde\gamma:=\alpha-d(\sum_{j=1}^Ng_j^2)$, which is clearly a D'Angelo form. A computation analogous to the one giving \eqref{quadratic_h} yields\begin{eqnarray*}
&&t\dbar \widetilde\gamma(L_p,\overline{L_p})-\left(\overline{\widetilde\gamma_{0,1}}\wedge \widetilde\gamma_{0,1}\right)(L_p,\overline{L_p})\\
&=&t\sum_{j=1}^N|L_pg_j|^2+t\dbar \alpha(L_p,\overline{L_p})-\left(\overline{\alpha_{0,1}}\wedge\alpha_{0,1}\right)(L_p,\overline{L_p}),\end{eqnarray*}
where $L_p\in\levinull_p$. By \eqref{claimed_cover_2} and \eqref{quadratic_h}, we see immediately that the form above is positive on $\distr$. This proves the desired bound $\norm(\widetilde\gamma;\distr)<t$ at the cost of an (unquantified) increase in size of the form. This is a drawback of the compactness argument employed, to which we now remedy.

Given $\varepsilon>0$, let $\chi_\varepsilon$ be a smooth function on $M$ with the following properties: \begin{enumerate}
	\item $\chi_\varepsilon$ is supported on an $\varepsilon$-neighborhood $S_\varepsilon$ of $S_\distr$;
	\item $0\leq\chi_\varepsilon\leq 1$ and $\chi_\varepsilon$ is identically equal to $1$ in an open neighborhood of $S_\distr$;
	\item  $||d\chi_\varepsilon||_\infty\leq C\varepsilon^{-1}$, where $C$ is uniform in $\varepsilon>0$.
\end{enumerate}
The emended form is $\gamma:=\alpha-d(\sum_{j=1}^N\chi_\varepsilon^2 g_j^2)$ for $\varepsilon$ small enough. In fact, we have $\norm(\gamma;\distr)=\norm(\widetilde\gamma;\distr)<t$, as a consequence of the second property above, while by the two other properties, \begin{eqnarray*}
\left\|d(\sum_{j=1}^N\chi_\varepsilon^2 g_j^2)\right\|_{\infty}&\leq&2\sum_{j=1}^N\left\|\chi_\varepsilon g_j^2d\chi_\varepsilon\right\|_{\infty}+2\sum_{j=1}^N\left\|\chi_\varepsilon^2 g_jd g_j\right\|_{\infty}\\
&\leq&2C\varepsilon^{-1}\sum_{j=1}^N\sup_{S_\varepsilon} |g_j|^2+2\sum_{j=1}^N\left\|d g_j\right\|_{\infty}\sup_{S_\varepsilon} |g_j|.
\end{eqnarray*}
Since $g_j$ vanishes on $S_\distr$, $\sup_{S_\varepsilon} |g_j|\leq C_j\varepsilon$ and one concludes that \begin{equation*}
\left\|\gamma\right\|_\levinull\leq \left\|\alpha\right\|_\levinull + \left\|d(\sum_{j=1}^N\chi_\varepsilon^2 g_j^2)\right\|_{\infty}\leq \left\|\alpha\right\|_\levinull+C'\varepsilon,
\end{equation*}
which is the desired bound on the $\levinull$-size of $\gamma$, for $\varepsilon$ small enough. This completes the proof of (A).

The proof of (B) is simpler. If $\distr_\lambda=0$ for some $\lambda<\lambda_1$, both sides of the identity equal $1$ and the statement is trivial. We can therefore assume that all the distributions $\distr_\lambda$ are nontrivial.

Fix  $t>\norm_K(\bigcap_\lambda\distr_\lambda)$. By definition, there exists a D'Angelo form $\alpha$ of $\levinull$-size $<K$ such that \begin{equation*}
\left(\overline{\alpha_{0,1}}\wedge\alpha_{0,1}\right)<t\dbar \alpha
\end{equation*} as quadratic forms on $\bigcap_\lambda\distr_\lambda$. Define $V(\alpha)$ as above. Since $\bigcap_\lambda\distr_\lambda$ is contained in $\Omega$, the complements $\C TM\setminus \distr_\lambda$, together with $\Omega$, form an open conical cover of $\C TM$. By Proposition \ref{compactness_prp} and the fact that the $\distr_\lambda$ are nested, it follows that there exists $\lambda_0$ such that $\distr_{\lambda_0}\subseteq \Omega$. This means that $\norm(\alpha;\distr_{\lambda_0})\leq t$ and therefore $\norm_K(\distr_{\lambda_0})\leq t$. By the arbitrariness of $t>\norm_K(\bigcap_\lambda\distr_\lambda)$, we conclude that $\inf_{\lambda<\lambda_1}\norm_K(\distr_\lambda)\leq \norm_K(\bigcap_\lambda\distr_\lambda)$. Since the reverse inequality is trivial, the conclusion follows.
\end{proof}

\section{Regularized Diederich--Forn\ae ss index}\label{df_sec}

As anticipated in the introduction, the "reduction to the core" theorem allows to prove the exact regularity of the $\dbar$-Neumann problem on smooth bounded pseudoconvex domains $\Omega\subset \C^n$ with trivial Levi core. For this, we need two ingredients: \begin{enumerate}
	\item  Harrington's generalization of a theorem of Kohn, giving a sufficient condition for exact regularity of the $\dbar$--Neumann problem; 
	\item  a slight generalization of Adachi--Yum theorem (that is, Theorem \ref{adachi_yum_thm_intro} of the introduction), involving a "regularized" version of the classical Diederich--Fornæss index, which we introduce below.
\end{enumerate}

Let us begin with the following pseudodistance on the space of defining functions.  

\begin{dfn}[cf. \cite{Harrington_global}, Definition 6.1] Let $\Omega\subset\C^n$ be a smooth bounded pseudoconvex domain. Given $r_1,r_2\in \mathrm{Def}(\Omega)$, let $f$ be the smooth function, defined on the common domain of $r_1$ and $r_2$, such that $r_2=e^fr_1$. We define \begin{equation}\label{def_metric}
\sigma(r_1,r_2):=\sup_{p\in M} \sup_{Z_p\in\levinull_p\colon |Z_p|=1}|Z_pf|,
\end{equation}
where $|Z_p|$ denotes the ordinary Euclidean norm of $Z_p$.
\end{dfn}

The pseudodistance $\sigma$ has an easy relation with the $(h, \mathcal{D})$-size introduced above. 

\begin{prp}\label{sigma_vs_size} Let $\Omega\subset\C^n$ be a smooth bounded pseudoconvex domain. If $\alpha_j$ ($j=1,2$) is a D'Angelo form associated to $r_j\in \mathrm{Def}(b\Omega)$ as in Proposition \ref{d'angelo_prp}, vi), then $\sigma(r_1,r_2)=||\alpha_2-\alpha_1||_\levinull$, where $||\cdot||_\levinull$ is the size of Definition \ref{hDsize_dfn} (the unspecified metric is the Euclidean one). 
\end{prp}

\begin{proof}
	Recall that $\alpha_j$ is not uniquely defined, but its restriction to the Levi-null distribution is unique (by Proposition \ref{d'angelo_prp}, iii)). If $r_2=e^fr_1$, then $d^cr_2=e^fd^cr_1$ on $M=b\Omega$. The conclusion follows from Proposition \ref{d'angelo_prp}, iv), and Definition \ref{hDsize_dfn}. 
	\end{proof}

The exact regularity theorem alluded to above is as follows. 

\begin{thm}[combination of \cite{Harrington_global}, Theorem 5.1 and Theorem 6.2, cf. main theorem of \cite{Kohn_quantitative}]\label{kohn_harrington_thm}
Let $\Omega\subset\C^n$ be a smooth bounded pseudoconvex domain. Assume that there exist $\delta_k>0$ ($k\in \N$) such that $\delta_k\rightarrow 1-$ and $r_k\in\mathrm{Def}(\Omega)$ such that $-(-r_k)^{\delta_k}$ is plurisubharmonic and \begin{equation}\label{harrington_condition}
\lim_{k\rightarrow+\infty}\sqrt{1-\delta_k}\sigma(r_k,r_0)=0,
\end{equation}
where $r_0$ is a fixed defining function. Then the $\dbar$-Neumann problem is exactly regular on $\Omega$.
\end{thm}

Notice that the condition \eqref{harrington_condition} is independent of the choice of $r_0$ and it holds in particular if the sequence $\{r_k\}$ is $\sigma$-bounded, that is, it is bounded as a subset of the pseudo-metric space $(\mathrm{Def}(\Omega), \sigma)$.

\medskip 

We can now recall the definition of Diederich--Fornæss index, and give a variant of this notion inspired by Theorem \ref{kohn_harrington_thm}. 

\begin{dfn}\label{df_dfn}	
Let $\Omega$ be a precompact smooth pseudoconvex domain in a complex manifold. A Diederich--Forn\ae ss (D--F in the sequel) exponent of $\Omega$, is any $\delta\in(0,1]$ for which there exists a defining function $r\in \mathrm{Def}(\Omega)$ with the property that $-(-r)^\delta$ is plurisubharmonic on $\Omega$ (that is, on the subset of $\Omega$ on which it is defined). We define the Diederich--Forn\ae ss index of $\Omega$, denoted by $\DF(\Omega)$, as the supremum of all D--F exponents of $\Omega$: \begin{equation*}
\DF(\Omega):=\sup\{\delta>0\colon \exists r\in \mathrm{Def}(\Omega) \text{ s.t. }-(-r)^\delta \text{ is plush. on }\Omega\}.
\end{equation*}

We say that $\Omega$ has regularized Diederich--Forn\ae ss index $1$ if there exist sequences $\delta_k\in (0,1]$ and $r_k\in\mathrm{Def}(\Omega)$ such that: \begin{enumerate}
	\item $\delta_k\rightarrow 1-$,
	\item $-(-r_k)^{\delta_k}$ is plurisubharmonic,
	\item $\{r_k\}_k$ is $\sigma$-bounded.
\end{enumerate}

\end{dfn}

We have the following immediate corollary of Theorem \ref{kohn_harrington_thm}. 

\begin{thm}\label{df1_thm}
If $\Omega\subset \C^n$ is a smooth bounded pseudoconvex domain with regularized D--F index $1$, then the $\dbar$-Neumann problem is exactly regular on $\Omega$.
\end{thm}

Finally, we state and prove our slightly improved Adachi--Yum theorem, involving the quantities $\norm_K$ of Definition \ref{adachi_yum_dfn} and the notion of "regularized D--F index $1$".

\begin{thm}\label{liu_yum_adachi_thm} Let $\Omega$ be a precompact smooth pseudoconvex domain in a complex manifold. Assume that there exists a smooth strictly plurisubharmonic function in a neighborhood of $b\Omega$. Then we have the identity
	\begin{equation*}
	\DF(\Omega)=\frac{1}{1+\norm(\levinull)}.
	\end{equation*}
Moreover, the domain $\Omega$ has regularized D--F index $1$ if and only if there exists $K<+\infty$ such that $\norm_K(\levinull)=0$.
\end{thm}

Putting together Theorem \ref{reduction_thm}, Theorem \ref{df1_thm}, and Theorem \ref{liu_yum_adachi_thm}, we obtain the desired corollary. 

\begin{thm}\label{exact_reg_thm}
	If $\Omega\subset \C^n$ is a smooth bounded pseudoconvex domain with trivial Levi core, then the $\dbar$-Neumann problem is exactly regular on $\Omega$.
	\end{thm}

We are left with the proof of Theorem \ref{liu_yum_adachi_thm}, which follows closely the argument of \cite{Adachi_Yum}. We take the occasion to give a presentation of this argument that stresses the invariant nature of the result, e.g., relying on Cartan's formula and avoiding as much as possible computations w.r.t. local coordinates and frames. We exploit the following two lemmas.  

\begin{lem}[cf. \cite{Yum_invariance}, Prop. 4.6 and \cite{Liu_index_I}, Lemma 2.1]
\label{liu_yum_key_lem} Let $M$ be a real hypersurface in a complex manifold. Let $N$ be a $(1,0)$-vector field defined in a neighborhood of $M$ such that $Nr\equiv 1$. Set $\alpha:=\alpha_{\theta, T}$, where $T=\frac{1}{2i}(N-\overline{N})$. If $Z$ is a $(1,0)$ vector field such that $Z_p\in\levinull_p$ at some point $p\in M$, then the identity\begin{equation*}
	N\left(	\partial\dbar r(Z,\overline{Z})\right)=\partial \alpha (Z,\overline{Z})-\frac{|\alpha(Z)|^2}{2}
	\end{equation*} holds at $p$.
\end{lem}

\begin{proof}
	Notice that $N\left(\partial\dbar r(gZ,\overline{gZ})\right)=|g|^2N\left(	\partial\dbar r(Z,\overline{Z})\right)$ for every smooth function $g$ and every $(1,0)$ vector field $Z$ such that $Z_p\in\levinull_p$. Thus, both sides of the identity are Hermitian forms on $\levinull_p$, and we may assume without loss of generality that $Z$ is a smooth $(1,0)$ vector field such that $Zr\equiv0$ on a neighborhood of $p$.
	
	The $2$-form $\partial\dbar r=d\dbar r$ is exact, thus closed. Cartan's formula yields\begin{eqnarray*}
		0&=&3d\left(\partial\dbar r\right)(N,Z,\overline{Z})\\
		&=&N\left(\partial\dbar r(Z,\overline{Z})\right)-Z\left(\partial\dbar r(N,\overline{Z})\right)+\overline{Z}\left(\partial\dbar r(N,Z)\right)\\
		&&-\partial\dbar r([N,Z], \overline{Z})+\partial\dbar r([N,\overline{Z}], Z)-\partial\dbar r([Z, \overline{Z}], N)
	\end{eqnarray*}	
	
	We now discuss each term separately. \newline
	By Proposition \ref{d'angelo_prp} and the fact that $\alpha$ is real, $Z\left(\partial\dbar r(N,\overline{Z})\right)=Z\left(\alpha(\overline{Z})\right)/2$. \newline
	Next, $\partial\dbar r(N,Z)\equiv0$ by type considerations, and thus $\overline{Z}\left(\partial\dbar r(N,Z)\right)\equiv0$. \newline
	Since both $Zr$ and $Nr$ are constant, $[N,Z]r=N(Zr)-Z(Nr)\equiv 0$. Thus, $[N,Z]\in T^{1,0}M$ and, since $Z_p\in \levinull_p$, $\partial\dbar r([N,Z], \overline{Z})$ vanishes at $p$. \newline
	The vector field $[N,\overline{Z}]$ is tangent to $M$ for the same reason as $[N,Z]$, but it has not definite type. At any rate, $[N,\overline{Z}]=g(N-\overline{N})\mod T^{1,0}M\oplus T^{0,1}M$, for some smooth function $g$. Applying $\partial r$ to this identity, we get $g=\partial r([N,\overline{Z}])$. Using again the fact that $Z_p\in \levinull_p$ and Proposition \ref{d'angelo_prp}, we get \begin{equation*}
	\partial\dbar r([N,\overline{Z}], Z)=-g\partial\dbar r(\overline{N},Z)=\frac{|\alpha(Z)|^2}{2}.
	\end{equation*}
	Finally, $\partial\dbar r([Z, \overline{Z}], N)=-\partial\dbar r(N,[Z, \overline{Z}]_{0,1})=-\alpha([Z, \overline{Z}]_{0,1})/2$.\newline
	Putting everything together, we get \begin{equation*}
	N\left(\partial\dbar r(Z,\overline{Z})\right)=\frac{Z\left(\alpha(\overline{Z})\right)}{2}-\frac{|\alpha(Z)|^2}{2}-\frac{\alpha([Z, \overline{Z}]_{0,1})}{2}.
	\end{equation*}
	To conclude, one uses Lemma \ref{complex_cartan_lem}.
\end{proof}

\begin{lem}\label{quad_form_lem}
	Let $U\subseteq\R^n$ be a neighborhood of the origin and $Q:U\times \C^N\rightarrow \R$ a $C^2$ function such that $\vec{a}\mapsto Q(x,\vec{a})$ is a quadratic form for every $x\in U$. Assume that: \begin{enumerate}
		\item $Q(x,\cdot)$ is nonnegative definite for every $x$ such that $x_n=0$;
		\item $\partial_{x_n}Q(0,\vec{a})>0$ for every nonzero $\vec{a}$ in the null space of $Q(0,\cdot)$.	
	\end{enumerate}
	Then there exists a neighborhood of the origin $V\subseteq U$ such that $Q(x,\cdot)$ is positive definite for every $x\in V\cap \{x_n>0\}$.
\end{lem}

\begin{proof}
	Assume without loss of generality that the null space of $Q(0,\cdot)$ is spanned by the first $k$ canonical basis vectors of $\C^N$. The second assumption and a simple compactness plus homogeneity argument gives \begin{equation}\label{cone_estimate}
	\partial_{x_n}Q(x,\vec{a})\geq \varepsilon|\vec{a}|^2 \qquad\forall x\in V,\ \forall \vec{a}\in \mathcal{C}_\varepsilon,
	\end{equation}
	where $\varepsilon>0$, $V$ is a neighborhood of the origin, and \begin{equation*}
	\mathcal{C}_\varepsilon:=\{\vec{a}=(a_1,\ldots, a_N)\in \C^N\colon  |a_{k+1}|^2+\cdots+|a_N|^2\leq \varepsilon\left(|a_1|^2+\cdots+|a_k|^2\right)\}.
	\end{equation*}
	Writing $x=(x', x_n)$, where $x'\in \R^{n-1}$ and $x_n>0$, Lagrange Theorem yields \begin{equation*}
	Q(x,\vec{a})=Q((x',0),\vec{a})+x_n\partial_{x_n}Q((x',x_*),\vec{a})
	\end{equation*} for some $x_*\in [0,x_n]$. By the first assumption and \eqref{cone_estimate}, we get $Q(x,\vec{a})>0$ for every $x\in V\cap \{x_n>0\}$ (for a possibly smaller neighborhood of the origin $V$) and $\vec{a}\in \mathcal{C}_\varepsilon$. Since $Q(0,\vec{a})\geq c|\vec{a}|^2$ for some $c>0$ and every $\vec{a}\in \C^N\setminus \mathcal{C}_\varepsilon$, another compactness plus homogeneity argument gives $Q(x,\vec{a})\geq c|\vec{a}|^2$ for every $x$ in a neighborhood of $0$ and $\vec{a}\in \C^N\setminus \mathcal{C}_\varepsilon$. This completes the proof.
\end{proof}

\subsection{Proof of the ``$\leq$'' and of the ``only if'' parts of Theorem \ref{liu_yum_adachi_thm}}

Since \begin{equation}\label{df_computation}
\partial\dbar \left(-(-r)^\delta\right)=\delta(-r)^{\delta-1}\left(	\partial\dbar r+(1-\delta)\frac{\partial r\wedge \dbar r}{-r}\right),
\end{equation} the D--F index is the supremum of the set of $\delta$'s such that $\delta\in (0,1]$ and \begin{equation*}
\left(\partial\dbar r+(1-\delta)\frac{\partial r\wedge \dbar r}{-r}\right)(Z,\overline{Z})\geq 0\qquad\forall Z\in T^{1,0}(\Omega\cap V)
\end{equation*} for some $r\in \mathrm{Def}(\Omega)$ and a neighborhood $V$ of the boundary. Fix $\delta$ and $r$ as above.

Let $N$ be as in Lemma \ref{liu_yum_key_lem} and $Z$ be a vector field of type $(1,0)$ such that $Z_p\in \levinull_p$ and $Zr\equiv0$, defined in a neghborhood of a point $p\in b\Omega$. The matrix of the Hermitian form $\partial\dbar r+(1-\delta)\frac{\partial r\wedge \dbar r}{-r}$, restricted to the span of $Z$ and $N$, is
\begin{equation}\label{matrix_L_N}
\begin{bmatrix}
\partial\dbar r(Z,\overline{Z})& \partial\dbar r(Z,\overline{N})\\
\partial\dbar r(N,\overline{Z}) & \partial\dbar r(N,\overline{N})+\frac{(1-\delta)}{-2r}
\end{bmatrix}
\end{equation}
Computing the determinant, we get the inequality
\begin{equation*}
(-r)\left(\partial\dbar r(Z,\overline{Z})\partial\dbar r(N,\overline{N})-|\partial\dbar r(Z,\overline{N})|^2\right)+(1-\delta)\frac{\partial\dbar r(Z,\overline{Z})}{2}\geq 0,
\end{equation*}
valid inside $\Omega$, in a neighborhood of $p$. Since the expression on the left hand side vanishes in $p\in b\Omega$, its derivative along the real vector $N+\overline{N}$ must be nonpositive at $p$. We get
\begin{equation*}
2|\partial\dbar r(Z,\overline{N})|^2+(1-\delta)\frac{(N+\overline{N})\partial\dbar r(Z,\overline{Z})}{2}\leq 0\qquad \text{at }p.
\end{equation*}
Proposition \ref{d'angelo_prp} and Lemma \ref{liu_yum_key_lem} yield
\begin{equation*}
\frac{|\alpha(Z)|^2}{2}+(1-\delta)\left(\frac{\partial \alpha (Z,\overline{Z})}{2}-\frac{\dbar \alpha (Z,\overline{Z})}{2}-\frac{|\alpha(Z)|^2}{2}\right)\leq 0\qquad \text{at }p,
\end{equation*}
or, equivalently, $\delta |\alpha(Z)|^2+ (1-\delta)\left(\partial \alpha (Z,\overline{Z})-\dbar \alpha (Z,\overline{Z})\right)\leq 0$. Here $\alpha=\alpha_{\theta, T}$, with $\theta$ and $T$ as in Proposition \ref{theta_def_fct}. Since $d\alpha(Z,\overline{Z})=0$ at $p$, this inequality may be rewritten as \begin{equation*}
\left(\overline{\alpha_{0,1}}\wedge \alpha_{0,1}\right) (Z,\overline{Z})\leq (\delta^{-1}-1)\dbar\alpha(Z,\overline{Z})\qquad \text{at }p.
\end{equation*}
By the arbitrariness of $p\in M$ and $Z_p\in \levinull_p$, $\widetilde\norm(\alpha;\levinull)\leq\delta^{-1}-1$ and thus (by Proposition \ref{norm_technical_prp}) $\norm(\levinull)\leq \delta^{-1}-1$. Taking the supremum over $\delta$, we get the inequality\begin{equation*}
\DF(\Omega)\leq \frac{1}{1+\norm(\levinull)}.
\end{equation*}

Assume now that $\Omega$ has regularized D--F index equal to $1$. Then the argument above proves that \begin{equation}\label{norm_alpha_k}
\widetilde\norm(\alpha_k;\levinull)\leq\delta_k^{-1}-1,
\end{equation} where $\delta_k\rightarrow1-$ and $\alpha_k$ is a sequence of D'Angelo forms associated to defining functions $r_k$ as in Definition \ref{df_dfn}. In particular, the $\sigma$-boundedness of $\{r_k\}_k$ translates, thanks to Proposition \ref{sigma_vs_size}, into the uniform boundedness in $\levinull$-size of $\{\alpha_k\}_k$. Thus, taking the supremum in \eqref{norm_alpha_k} and recalling again Proposition \ref{norm_technical_prp}, we conclude that $\norm_K(\levinull)=0$ for $K$ large enough.

\subsection{Proof of the ``$\geq$'' and of the ``if'' parts of Theorem \ref{liu_yum_adachi_thm}}

Assume that $t>\norm(\levinull)$. This means that there exists a D'Angelo form $\alpha$ such that
\begin{equation}\label{DF_s>0}
\left(\overline{\alpha_{0,1}}\wedge \alpha_{0,1}\right)(Z_p,\overline{Z_p})<t\dbar \alpha(Z_p,\overline{Z_p}) \quad \forall p\in M,\ Z_p\in \levinull_p.
\end{equation}

By Proposition \ref{d'angelo_prp}, vi), we may assume that $\alpha(Z)=2\partial\dbar r(Z,\overline{N})$, where $r\in\mathrm{Def}(\Omega)$ and $N$ is a $(1,0)$-vector field defined on a neighborhood of $M$ such that $Nr\equiv 1$.

We claim that $-(-r)^{\frac{1}{1+t}}$ is strictly plurisubharmonic on $\Omega\cap V$ where $V$ is a neighborhood of $b\Omega$. From this the inequality $\DF(\Omega)\geq \frac{1}{1+t}$, and hence the thesis, follows immediately. By \eqref{df_computation}, the claim is equivalent to \begin{equation*}
\left(\partial\dbar r+\frac{t}{1+t}\frac{\partial r\wedge \dbar r}{-r}\right)(X,\overline{X})> 0\qquad\forall X\in T^{1,0}(\Omega\cap V).
\end{equation*}
It is clearly enough to fix $p\in M$ and prove that there is a neighborhood $V$ of $p$ such that the estimate above holds. Equivalently, we want the matrix \eqref{matrix_L_N} to be positive definite on $\Omega \cap V$ for every $(1,0)$-vector field $Z$ such that  $Zr\equiv 0$ on $V$. Notice that $ \partial\dbar r(N,\overline{N})+\frac{t}{-2(1+t)r}$ is strictly positive on $\Omega\cap V$, if $V\subseteq \{r>-\varepsilon\}$ with $\varepsilon>0$ small enough. Thus, by Sylvester's criterion, it is enough to have
\begin{equation*}
(-r)\left(\partial\dbar r(Z,\overline{Z})\partial\dbar r(N,\overline{N})-|\partial\dbar r(Z,\overline{N})|^2\right)+\frac{t}{1+t}\frac{\partial\dbar r(Z,\overline{L})}{2}>0\qquad\forall Z\colon Zr\equiv 0,
\end{equation*} on $\Omega\cap V$. Let $Z_1,\ldots, Z_{n-1}$ be a local basis of $\ker(\partial r)$. Writing $Z=\sum_{j=1}^{n-1}a_jZ_j$, the expression above induces a quadratic form $Q(q;\vec{a})$ on $\C^{n-1}$, depending smoothly on a point $q$ in a neighborhood of $p$. Notice that, since $\Omega$ is pseudoconvex, $Q(q,\cdot)$ is nonnegative definite for every $q\in M$. Moreover, as in the proof of the first half of the theorem, one sees that \begin{equation*}
(N+\overline{N})Q(\cdot\ ;\vec{a})_{|p}=\left(\frac{1}{1+t}\overline{\alpha_{0,1}}\wedge \alpha_{0,1}-\frac{t}{1+t}\dbar \alpha\right)(Z,\overline{Z}).
\end{equation*} for every $Z=\sum_{j=1}^na_jZ_j$ in the null space of the Levi form at $p$. In virtue of \eqref{DF_s>0}, we are in a position to apply Lemma \ref{quad_form_lem} (with $x_n=-r$) and conclude that $Q(q,\cdot)$ is positive definite for every $q\in \Omega$ sufficiently near $p$. This completes the proof of the identity of Theorem \ref{liu_yum_adachi_thm}.

Assume now that $\norm_K(\levinull)=0$ for some $K<+\infty$. The argument above shows that there exists a sequence $\delta_k=\frac{1}{1+t_k}$ with $t_k\rightarrow0+$, and a sequence $r_k\in\mathrm{Def}(\Omega)$ such that $-(-r_k)^{\delta_k}$ is plurisubharmonic and the $\levinull$-sizes of the associated D'Angelo forms $\alpha_k$ are uniformly bounded. Recalling Proposition \ref{sigma_vs_size}, this implies that $\Omega$ has regularized D--F index $1$.

\appendix

\section{A quantitative bound for $\norm_X$}

In this appendix, we provide a quantitative estimate of the function $\norm_X$ of Section \ref{de_rham_sec}, in the case where $X$ is a pseudoconvex domain in $\C^n$ satisfying an additional regularity property. Our estimate is a sort of interpolation inequality comparing $\norm_\Omega$ with two other, possibly more natural, norm-like functions on the first de Rham cohomology of $\Omega$. \medskip 

We use the following version of John--Nirenberg inequality, appearing in \cite{Hurri_Syrjaenen}. A domain $\Omega\subseteq\R^N$ is said to satisfy a quasi-hyperbolic boundary condition with constants $a, b>0$ if for some $x_0\in \Omega$ the inequality \[
k_\Omega(x,x_0)\leq a\log\frac{1}{d(x,b\Omega)}+b
\] holds for every $x\in \Omega$, where:\begin{itemize}
	\item $d(\cdot,b\Omega)$ is the Euclidean distance to the boundary,
	\item $k_\Omega(x,y)$ is the distance w.r.t. the Riemannian metric $d(x,b\Omega)^{-2}dx^2$. 
	\end{itemize}

\begin{thm}\label{john_nirenberg_thm}[Theorem 3.5 of \cite{Hurri_Syrjaenen}]
	Let $\Omega\subset\R^n$ satisfy a quasi-hyperbolic boundary condition with constants $a$ and $b$. There are positive constants $c_1=c_1(n,a)$ and $c_2=c_2(n,a)$ such that if $u\in L^1(\Omega,\R)$ satisfies the BMO condition \begin{equation*}
	[u]_{\mathrm{BMO}}:=\sup_{B}\frac{1}{|B|}\int_B|u-u_B|<+\infty,
	\end{equation*}
	where the supremum is over all Euclidean balls $B=B(x,r)$ such that $B(x,\frac{9}{8}r)\subseteq\Omega$, then \begin{equation*}
	\int_{\Omega}\exp\left(\frac{c_1|u-u_\Omega|}{[u]_{\mathrm{BMO}}}\right)\leq c_2|\Omega|.
	\end{equation*}
\end{thm}

\begin{thm}
Let $\Omega\subseteq\C^n$ be a pseudoconvex domain and $[\alpha]\in H^1_{\mathrm{dR}}(\Omega)$ such that $\norm_\Omega([\alpha])<+\infty$ (as defined in Section \ref{de_rham_sec}). If $\Omega_1\subseteq\Omega$ is a relatively compact pseudoconvex domain satisfying the quasi-hyperbolic boundary condition with constants $a$ and $b$, 	
\begin{equation*}	
\norm_{L^1(\Omega_1)}([\alpha])^2\leq c(n,a) \norm_\Omega([\alpha])\norm_{L^\infty(\Omega_1)}([\alpha])\qquad\forall[\alpha]\in H^1_{\mathrm{dR}}(\Omega), 
\end{equation*} where \begin{equation*}
\norm_{L^1(\Omega)}([\alpha]):=\inf_{h\in \PH(\Omega)\colon [d^ch]=[\alpha]}\frac{1}{|\Omega|}\int_\Omega |h|
\end{equation*} and
\begin{equation*}
\norm_{L^\infty(\Omega)}([\alpha]):=\inf_{h\in \PH(\Omega)\colon [d^ch]=[\alpha]}||h||_{L^\infty(\Omega)}.
\end{equation*}
	\end{thm}

The quantities $\norm_{L^1(\Omega)}([\alpha])$ and $\norm_{L^\infty(\Omega)}([\alpha])$ are well-defined thanks to Lemma \ref{hodge_theory_lem}. 

\begin{proof} Write $\norm$ in place of $\norm_\Omega$ for simplicity. Let $\norm([\alpha])<\lambda$ and $h\in \PH(\Omega)$ such that $[\alpha]=[d^ch]$, whose existence is guaranteed by Lemma \ref{hodge_theory_lem}. Then, by \eqref{norm_pluriharm}, there exists $f\in C^\infty(\Omega,\R)$ such that \begin{equation*}
 \partial (h+if)\wedge\dbar(h-if)\leq\lambda \partial \dbar f.
\end{equation*}
Tracing this inequality and using a trivial bound, we get \begin{equation*}
\frac{|\nabla f|^2}{2}\leq 2\lambda \Delta f+|\nabla h|^2,
\end{equation*}
where $\nabla$ and $\Delta$ are the Euclidean gradient and Laplacian. Multiplying by $\eta^2\leq 1$ with $\eta$ test function supported on $B(x,\frac{17}{16}r)\subset B(x,\frac{9}{8}r)\subseteq\Omega_1$, equal to $1$ on $B(x,r)$, and such that $|\nabla \eta|\leq 17r^{-1}$, we find
\begin{eqnarray*}
	\frac{1}{2}\int_\Omega\eta^2|\nabla f|^2&\leq& -4\lambda \int_\Omega\eta\nabla f\cdot \nabla\eta+\int_\Omega\eta^2|\nabla h|^2\\
	&\leq& \frac{1}{4}\int_\Omega\eta^2|\nabla f|^2+16\lambda^2 \int_\Omega|\nabla \eta|^2+\int_{B(x,\frac{17}{16}r)}|\nabla h|^2.
\end{eqnarray*}
Since $h$ is harmonic, we have the standard inequality
\begin{equation*}
\int_{B(x,\frac{17}{16}r)}|\nabla h|^2\leq C_n'r^{-2}\int_{B(x,\frac{9}{8}r)}|h|^2\leq C_nr^{n-2}||h||_{L^\infty(\Omega_1)}^2,
\end{equation*}
which gives\begin{equation*}
\int_{B(x,r)}|\nabla f|^2\leq C_nr^{2n-2}(\lambda^2+||h||_{L^\infty(\Omega_1)}^2).
\end{equation*}
Thus, by H\"older and Poincar\'e inequalities, if $B=B(x,r)$ is a Euclidean ball such that $B(x,\frac{9}{8}r)\subseteq\Omega_1$
\begin{equation*}
\frac{1}{|B|}\int_B|f-f_B|\leq C_n(\lambda+||h||_{L^\infty(\Omega_1)}).
\end{equation*}
By Theorem \ref{john_nirenberg_thm}, setting $g:=f-f_{\Omega_1}$ we have\begin{equation*}
\int_{\Omega_1}\exp\left(\frac{c_1 |g|}{\lambda+||h||_{L^\infty(\Omega_1)}}\right)\leq c_2|\Omega_1|,
\end{equation*} where $c_1, c_2>0$ depend only on $n$, $a$ and $b$. Of course, we have
\begin{equation}\label{epsilon_bound} \partial (h+ig)\wedge\dbar(h-ig)\leq\lambda \partial \dbar g.\end{equation}
The $(0,1)$-form $\dbar(h-ig)$ is clearly $\dbar$-closed. By \eqref{epsilon_bound} and H\"ormander theorem, there exists $u$ such that $\dbar u =\dbar(h-ig)$ and \begin{equation*}
\int_{\Omega_1}|u|^2\exp\left(-\frac{c_1 g}{\lambda+||h||_{L^\infty(\Omega_1)}}\right)\leq
\frac{\lambda(\lambda+||h||_{L^\infty(\Omega_1)})}{c_1} \int_{\Omega_1}\exp\left(-\frac{c_1 g}{\lambda+||h||_{L^\infty(\Omega_1)}}\right).
\end{equation*}
Thus \begin{eqnarray*}
	\left(\int_{\Omega_1}|u|\right)^2&\leq& \int_{\Omega_1}|u|^2\exp\left(-\frac{c_1 g}{\lambda+||h||_{L^\infty(\Omega_1)}}\right)\int_{\Omega_1}\exp\left(\frac{c_1 g}{\lambda+||h||_{L^\infty(\Omega_1)}}\right)\\
	&\leq&\frac{\lambda(\lambda+||h||_{L^\infty(\Omega_1)})}{c_1}\left\{\int_{\Omega_1}\exp\left(\frac{c_1 |g|}{\lambda+||h||_{L^\infty(\Omega_1)}}\right)\right\}^2\\
	&\leq&c_2^2c_1^{-1} \lambda(\lambda+||h||_{L^\infty(\Omega_1)})|\Omega_1|^2.
\end{eqnarray*} Since $u +F=h-ig$ for some holomorphic $F:\Omega_1\rightarrow\C$, we have $h-\Re(F)=\Re(u)$ and \begin{equation*}	
	\left(\int_{\Omega_1}|h-\Re(F)|\right)^2\leq c_2^2c_1^{-1} \lambda(\lambda+||h||_{L^\infty(\Omega_1)})|\Omega_1|^2
	\end{equation*}
Since $h$ is any pluriharmonic function such that $[d^ch]=[\alpha]$, the same estimate holds for $h+\Re(G)$, where $G$ is any holomorphic function on $\Omega_1$. Therefore, we get
\begin{equation*}	
\norm_{L^1(\Omega_1)}([\alpha])^2\leq c_2^2c_1^{-1} \lambda(\lambda+\norm_{L^\infty(\Omega_1)}([\alpha])).
\end{equation*}
Since $\lambda>\norm([\alpha])$ is arbitrary, we get
\begin{equation*}	
\norm_{L^1(\Omega_1)}([\alpha])^2\leq c_2^2c_1^{-1} \norm([\alpha])(\norm([\alpha])+\norm_{L^\infty(\Omega_1)}([\alpha])).
\end{equation*}
If $\norm([\alpha])\leq\norm_{L^\infty(\Omega_1)}([\alpha])$, we get the thesis. Otherwise, it follows from the trivial inequality $\norm_{L^1(\Omega_1)}([\alpha])\leq \norm_{L^\infty(\Omega_1)}([\alpha])$.
\end{proof}

\bibliography{paper}
\bibliographystyle{alpha}
\end{document}